\documentclass[11pt,a4paper]{article}

\usepackage{subfigure}
\usepackage{amssymb}
\usepackage{amsthm}
\usepackage{amsmath}
\usepackage{enumerate}
\usepackage{float}
\usepackage{graphicx}

\newtheorem{theorem}{Theorem}[section]
\newtheorem{definition}[theorem]{Definition}
\newtheorem{lemma}[theorem]{Lemma}
\newtheorem{notation}[theorem]{Notation}

\newcommand{\C}{\mathbb{C}}

\begin{document}
\begin{center}
{\LARGE\bf Isometric coactions\\ of compact quantum groups\\
           on compact quantum metric spaces}
\bigskip

{\sc by Johan Quaegebeur$^{(1)}$ and Marie Sabbe$^{(1)}$
\setcounter{footnote}{1}
\footnotetext{Department of Mathematics; K.U.Leuven; Celestijnenlaan 200B; B--3001 Leuven (Belgium).\\
E-mails: johan.quaegebeur@wis.kuleuven.be and marie.sabbe@wis.kuleuven.be}}
\end{center}

\begin{abstract}
We propose a notion of isometric coaction of a compact quantum group on a compact quantum metric space in the framework of Rieffel where the metric structure is given by a Lipnorm. We prove the existence of a quantum isometry group for finite metric quantum spaces, preserving a given state.
\end{abstract}

\section{Introduction}

This paper pertains to the study quantum symmetries of a (classical or quantum) space. This problem can be formulated and studied in various settings. In \cite{Wang-Quantum_symmetry_groups_of_finite_spaces}, Wang considers the case where the space is finite and doesn't carry any extra structure. The quantum symmetry groups he obtains, can thus be interpreted as quantum permutation groups.

The spaces we are interested in in this paper, are metric spaces, both classical and quantum. Symmetries of a (quantum) metric space should then preserve the extra metric structure of the space, i.e.\ they should be `isometric' in a sense to be made precise. For finite metric spaces, the `quantum isometry group' is therefore expected to be a quantum subgroup of the quantum permutation group of Wang.
For finite classical metric spaces, this problem was studied by Banica \cite{Banica-Quantum_automorphism_groups_of_small_metric_spaces}. He has given a definition for a quantum symmetry of a classical finite metric space. With this definition, he was able to construct a quantum isometry group as a subgroup of Wang's quantum permutation group.

The framework of Banica is still a half-classical framework since the metric spaces he considers, are classical. The concept of a metric space also has a quantum version. One approach of quantizing metric spaces is by spectral triples (\cite{Connes-Compact_metric_spaces_Fredholm_modules_and_hyperfiniteness}). These triples are used in non-commutative geometry to quantize the differential structure of a manifold as well as the metrical structure of the space. The quantum isometries of spectral triplets have been studied in a number of papers (e.g.\ \cite{Goswami-Quantum_group_of_isometries_in_classical_and_noncommutative_geometry},
\cite{Bhowmick&Goswami-Quantum_group_of_orientation-preserving_Riemannian_isometries}, \cite{Goswami-Quantum_isometry_group_for_spectral_triples_with_real_structure}, \cite{Bhowmick&Goswami-Quantum_isometry_groups_examples_and_computations},
\cite{Bhowmick&Goswami-Quantum_isometry_groups_of_the_Podles_Spheres},
\cite{Banica&Goswami-Quantum_isometries_and_noncommutative_spheres},
\cite{Bhowmick&Skalski-Quantum_isometry_groups_of_noncommutative_manifolds_associated_to_group_C*-algebras}).

There exists another framework to describe quantum metric spaces. 
In \cite{Rieffel-Metrics_on_state_spaces}, Rieffel introduces a framework in which he  only quantizes the metrical information of the space, disregarding the differential structure. Hence Rieffel considers quantum spaces that come only with a quantum metric (encoded by a so called Lipnorm) and carry no further structure. This is the framework we will be working in. In fact, this paper wants to formulate a suitable answer to Rieffel's suggestion at the end of section 6 of \cite{Rieffel-Gromov-Hausdorff_distance_for_quantum_metric_spaces}: ``It would be interesting to develop and study the notion of a `quantum isometry group' for quantum metric spaces as quantum subgroups of the quantum symmetry groups studied by Wang \cite{Wang-Quantum_symmetry_groups_of_finite_spaces}''.

The quantum metric spaces considered by Rieffel, are compact. Classically, the isometry group of a compact space is a compact group.
Hence, if we want to define a `quantum isometry group' in Rieffel's framework, we might expect it to be compact. Therefore, we gather some information on compact quantum groups and compact quantum metric spaces together with other preliminary notions in the second section.

Next, we introduce and motivate our notion of isometric actions in the third section. Actually, we will need two notions: 1-isometric coactions and full-isometric coactions.  The former notion is the more natural one from the intuitive point of view. The latter is technically stronger and will allow us  to construct a quantum isometry group in some cases. In the fourth section, we prove that both notions coincide with the existing definition of Banica \cite{Banica-Quantum_automorphism_groups_of_small_metric_spaces} for quantum isometries of finite classical metric spaces.

In the following sections, we further explore the notion(s) of isometric coactions we introduced. In the fifth section we mainly prove Theorem \ref{theorem largest quantum subgroup acting isometrically} which is a quantum version of the classical result that a group acting on a metric space has a largest subgroup acting isometrically. Using that theorem, we can prove in section 6 the main result of this paper: the existence of a quantum isometry group for finite quantum metric spaces preserving a given state on the space, as subgroups of the quantum symmetry
groups of Wang.

\section{Preliminaries}

First we need to recall some introductory definitions. We define compact quantum groups, compact quantum metric spaces and coactions.

\subsection{Compact quantum groups}

\begin{definition}
A \textbf{compact quantum group (CQG)} is a pair $(A,\Delta)$ where
\begin{enumerate}
\item $A$ is a unital C*-algebra
\item the 'comultiplication' $\Delta: A \rightarrow A \otimes A$ is a *-homomorphism
such that
\begin{itemize}
\item $\Delta$ is 'coassociative': $(\iota \otimes \Delta) \; \Delta = (\Delta \otimes \iota) \; \Delta$
\item the sets $\Delta(A) (1 \otimes A)$ and $\Delta(A) (A \otimes 1)$ are dense in $A \otimes A$.
\end{itemize}
\end{enumerate}
\end{definition}

The notion of a compact quantum group generalizes the notion of a `classical' compact group. Indeed, any compact group $G$
can be seen as a CGQ: take $A=C(G)$ (the commutative C*-algebra of continuous complex valued functions on $G$)
and define the comultiplication
\[\Delta: C(G)  \rightarrow  C(G) \otimes C(G) \cong C(G \times G): f\mapsto \Delta f\]
by
\begin{equation}\label{equation classical comultiplication} \\
(\Delta f)(s,t) =f(st) \ \mbox{ for }s,t\in G.
\end{equation}
Conversely, if $(A,\Delta)$ is a CQG for which the C*-algebra $A$ happens  to be commutative, then there is a compact group $G$
such that $A$ can be identified with $C(G)$ and such that, under this identification, $\Delta$ is given by (\ref{equation classical comultiplication}).

\begin{definition} \label{definition morphism of CQGs}
Let $(A, \Delta)$ and $(\tilde{A}, \tilde{\Delta})$ be two compact quantum groups. We say that
$\varphi: A \rightarrow \tilde{A}$ is a \textbf{morphism of CQGs} from $(A, \Delta)$ to $(\tilde{A}, \tilde{\Delta})$
if $\varphi$ is a *-homomorphism such that
\[(\varphi \otimes \varphi) \Delta = \tilde{\Delta} \varphi.\]
\end{definition}

For more information on compact quantum groups we refer to \cite{Maes&VanDaele-Notes_on_CQGs},\cite{Woronowicz-CQGs_symetries_quantiques}.

\subsection{Compact quantum metric spaces}

\begin{definition} \label{definition CQMS}
A \textbf{compact quantum metric space (CQMS)} is a pair $(B,L)$ where
\begin{itemize}
\item[(i)] $B$ is a unital C*-algebra
\item[(ii)] $L$ is a \textbf{Lipnorm} on $B$, i.e.\ $L: B \rightarrow [0,+\infty]$ is a seminorm such that
\begin{itemize}
\item $L(b) = L(b^*)$ for every $b \in B$
\item $\forall b \in B: L(b) = 0 \Leftrightarrow b \in \mathbb{C} 1$
\item $L$ is lower semicontinuous
\item the $\rho_L$-topology coincides with the weak *-topology on $\mathcal{S}(B)$, where $\rho_L$ is the metric on the
state space $\mathcal{S}(B)$ defined by
\[\rho_L(\mu, \nu) = \sup\bigl\{|\mu(b)-\nu(b)| \bigm| b \in B, L(b) \leq 1\bigr\}\]
for every $\mu, \nu \in \mathcal{S}(B)$.
\end{itemize}
\end{itemize}
\end{definition}

The classical notion of a compact metric space fits into this framework.
Indeed, let $(X,d)$ be a compact metric space. Put $B = C(X)$ and consider the Lipschitz seminorm
\[L_d: C(X) \rightarrow [0,+\infty]: f \mapsto \sup\left\{ \frac{|f(x)-f(y)|}{d(x,y)} \biggm|  x,y \in X, x \neq y\right\}.\]
One can prove that this seminorm is a Lipnorm.

Conversely, if $(B,L)$ is a CQMS for a commutative C*-algebra $B$, then $B=C(X)$ for some compact space $X$ and
one can construct a metric $d$ on $X$ by restricting the metric $\rho_{L}$
on the state space to the set of pure states $\mu_x: C(X) \rightarrow \mathbb{C}: f \mapsto f(x)$ (where $x \in X$).
Rieffel has proven in \cite{Rieffel-Metrics_on_state_spaces} that $L$ is now equal to the Lipschitz seminorm $L_d$.

Another source of examples of compact quantum metric spaces is given by some spectral triples.
A spectral triple $(\mathcal{B}, \mathcal{H}, D)$ is a *-subalgebra $\mathcal{B}$ of the bounded operators on a Hilbert space $\mathcal{H}$, together with a Dirac operator $D$ on $\mathcal{H}$. This is a self-adjoint
(possibly unbounded) operator on $\mathcal{H}$ such that $[D,b]$ has a bounded extension for every $b \in \mathcal{B}$.
In some cases, the normclosure of $\mathcal{B}$ in $\mathcal{B(H)}$ will be a
CQMS for the seminorm $L$ determined by $L(b) = ||[D,b]||$ for $b\in\mathcal{B}$.

For more information and examples of compact quantum metric spaces, we refer to the work of
Rieffel \cite{Rieffel-Metrics_on_state_spaces}, \cite{Rieffel-Compact_quantum_metric_spaces}. Note that in \cite{Rieffel-Compact_quantum_metric_spaces} and later articles,
the definition of a Lipnorm is without the requirement of lower semicontinuity. But Rieffel shows that, starting from any,
possibly not lower semicontinuous, Lipnorm $L$ on $B$ , one can always construct a lower semicontinuous Lipnorm $\tilde L$
such that $L$ and $\tilde L$ induce the same metric on the state space of $B$.
Therefore, we prefer to have this requirement in the definition.

\subsection{Coactions}

\begin{definition} \label{definition coaction}
A \textbf{coaction} of a CQG $(A,\Delta)$ on a unital C*-algebra $B$ is a unital *-homomorphism $\alpha: B \rightarrow B \otimes A$ such that
\begin{itemize}
\item[(i)]
$(\iota \otimes \Delta) \; \alpha = (\alpha \otimes \iota) \; \alpha$,
\item[(ii)] the set $\alpha(B) \; (1 \otimes A)$ is norm dense in $B \otimes A$.
\end{itemize}
We say that $\alpha$ is \textbf{faithful} if the set $\{(\psi \otimes \iota) \alpha(b) \; | \; b \in B, \psi \in B^*\}$ generates $A$ as a C*-algebra.
\end{definition}

A classical action of a compact group $G$ on a space $X$ fits in this framework by taking
$\alpha: C(X) \rightarrow C(X) \otimes C(G) \cong C(X \times G):f\mapsto \alpha(f)$
determined by
$$\alpha(f)(x,s)=f(s\cdot x) \ \mbox{ for }x\in X, s\in G.$$

\begin{definition} \label{definition morphism of QTGs}
If $B$ is a C*-algebra, $(A, \Delta)$ is a compact quantum group and $\alpha: B \rightarrow B \otimes A$ is a coaction, we call the triple $(A, \Delta, \alpha)$ a \textbf{quantum transformation group (QTG)} of $B$. We say that the QTG $(A, \Delta, \alpha)$ is faithful if the coaction $\alpha$ is faithful.

If $\psi$ is a functional on $B$, we say that $\alpha$ preserves $\psi$ if
$$(\psi \otimes \iota) \alpha(b) = \psi(b) 1_A$$
for all $b \in B$. In that case, we say that $(A, \Delta, \alpha)$ is a \textbf{quantum transformation group of the pair $(B, \psi)$}.

If $(A, \Delta, \alpha)$ and $(\tilde{A}, \tilde{\Delta}, \tilde{\alpha})$ are two quantum transformation groups of $B$ or $(B, \psi)$, we say that $\varphi: A \rightarrow \tilde{A}$ is a \textbf{morphism of QTGs}
from $(A, \Delta, \alpha)$ to $(\tilde{A}, \tilde{\Delta}, \tilde{\alpha})$ if $\varphi$ is a morphism of compact quantum groups
such that \[(\iota \otimes \varphi) \alpha = \tilde{\alpha}.\]
\end{definition}

\section{Isometric coactions}

Now the natural question arises: if in the Definition \ref{definition coaction} $B$ happens
to have the extra structure of a compact quantum metric space as in Definition \ref{definition CQMS},
when would it make sense to call the action $\alpha$ isometric?

Before we answer this question in general, we take a look at a specific case.
Consider the situation where the CQG is actually a classical group, i.e. the situation where $A$
is $C(G)$ for some compact group $G$.
Let $\alpha: B \rightarrow B \otimes C(G)$ be a coaction of the group $G$ on a compact quantum metric space $(B,L)$.
To simplify notations, we identify $B \otimes C(G)$ with $C(G,B)$, by
\[b \otimes f: G \rightarrow B: s \mapsto f(s) b\]
for $b \in B$ and $f \in C(G)$. In this context, it is rather evident what an isometric coaction should be:
we call $\alpha$ \textit{isometric} if
\begin{equation}\label{definition isometric action of CG on CQMS}
\forall b \in B, \; \forall s \in G: \ L(\alpha(b)(s)) = L(b).
\end{equation}
Notice that, in case of a classical CQMS $(B,L)$, i.e. when $B=C(X)$ and $L=L_d$ for some compact metric
space $(X,d)$, condition (\ref{definition isometric action of CG on CQMS}) expresses that the action of $G$ on $X$ is isometric in the classical sense.

Now we  want to extend this definition to the double-quantum case, where a CQG is acting on a CQMS.
Therefore we must get rid of everything that refers to the individual group elements. To start with,
we write $\alpha(b)(s)$ as $(\iota \otimes \omega_s) \alpha(b)$ where $\omega_s$ is the pure state on $C(G)$ given by
evaluating in $s$:
\begin{equation} \label{definition states omega_s} \omega_s: C(G) \rightarrow \mathbb{C}: f \mapsto f(s).\end{equation}
To get rid of the $s$ in this formula,
we want to replace $\omega_s$ by an arbitrary state $\omega$ on $A$.

Let's first take a look at what happens for a convex combination $\omega = \sum_{i=1}^k \lambda_i \omega_{s_i}$
for some group elements $s_1, \cdots, s_k \in G$ and some positive numbers $\lambda_1, \cdots, \lambda_k$ with $\sum_{i=1}^k \lambda_i = 1$. Then, if the equality $L((\iota \otimes \omega_s) \alpha(b)) =
L(b)$ holds for every $s \in G$ and every $b \in B$, we have that
\[L((\iota \otimes \omega) \alpha(b)) \leq \sum_{i=1}^k \lambda_i L((\iota \otimes \omega_{s_i}) \alpha(b)) =
L(b)\] for all $b \in B$. Using the lower semicontinuity of $L$, one can show that the above inequality also holds for arbitrary states, which will be done in Lemma \ref{theorem inequality 1-isometric sufficient for weak*-dense convex set of states}.
These ideas motivate the following definition for a coaction to be isometric. For later purposes, we formulate the definition in a more general setting.

\begin{definition} \label{definition 1-isometric *-homomorphism}
For a unital C*-algebra $A$ and a CQMS $(B,L)$, \\ a *-homomorphism $\alpha: B \to B \otimes A$ is called \textbf{1-isometric} if and only if
\begin{equation} \label{inequality 1-isometric} \forall \omega \in \mathcal{S}(A),\;\forall b \in B: L((\iota \otimes \omega) \alpha(b)) \leq L(b).\end{equation}
\end{definition}

When writing this paper, we discovered that this notion already appeared in a recent paper of H.~Li (\cite{Li-CQMS_and_ergodic_actions_of_CQGs}, definition 8.8).
He calls a coaction invariant if the above inequality holds. However, he did not further explore this notion
in the direction we are studying (the problem of the existence of a quantum isometry group).

Let us explain why we used the terminology '1-isometric', and why the definition was formulated for *-homomorphisms instead of coactions. This has to do with the following problem: suppose we let two CQGs $(A_1, \Delta_1)$ and $(A_2, \Delta_2)$ act 1-isometrically on a CQMS $(B,L)$ (denote the coactions by $\alpha_1$ and $\alpha_2$). In order to have a good definition of isometric coactions, we want the 'coaction' obtained by applying both coactions successively (i.e. the *-homomorphism $(\alpha_1 \otimes \iota) \alpha_2$), to be isometric. Firstly, the *-homomorphism $(\alpha_1 \otimes \iota) \alpha_2$ is not necessarily a coaction for the comultiplication one can define on $A_1 \otimes A_2$. That is why we will also be using the terminology `1-isometric' for *-homomorphisms $\alpha$ from a CQMS $(B,L)$ to $B \otimes A$ where $A$ is a C*-algebra. Moreover, we cannot prove why condition (\ref{inequality 1-isometric}) would hold for $(\alpha_1 \otimes \iota) \alpha_2$. To solve this problem, we strengthen the notion of 1-isometric to `full isometric'. In order to formulate the definition, we introduce some notations.

\begin{notation} \label{notation *-product of coactions}
Let $B$ be a C*-algebra and $(C_i, \Delta_i, \beta_i)$ be QTGs of $B$ for $i=1, \cdots, n$. Then we denote $\beta_1 \ast \beta_2$ for the *-homomorphism $(\beta_1 \otimes \iota) \beta_2: B \to B \otimes C_1 \otimes C_2$. Inductively we can define $\beta:= \beta_1 \ast \cdots \ast \beta_n: B \to B \otimes C_1 \otimes \cdots \otimes C_n$.
\end{notation}

\begin{definition}[Full-isometric coaction] \label{definition full-isometric coaction}
A coaction $\alpha$ of a CQG $(A, \Delta)$ on a CQMS $(B,L)$ is called \textbf{full-isometric} if and only if for all faithful QTGs $(C_i, \Delta_i, \beta_i)$ ($i = 1, \cdots, n$) such that $\beta:= \beta_1 \ast \cdots \ast \beta_n$ is 1-isometric, we have that $\alpha \ast \beta$ is 1-isometric, or
$$ \forall \omega \in \mathcal{S}(A \otimes C_1 \otimes \cdots \otimes C_n),\;\forall b \in B: L((\iota \otimes \omega)(\alpha \otimes \iota) \beta(b)) \leq L(b).$$
\end{definition}

As the terminology suggests, being full-isometric implies being 1-isometric. 

The inequality (\ref{inequality 1-isometric}) is not always very practical to work with. The following lemma states that it is enough to have inequality (\ref{inequality 1-isometric}) for some specific states, for example the set of all pure states.

\begin{lemma} \label{theorem inequality 1-isometric sufficient for weak*-dense convex set of states}
Suppose we have two unital C*-algebras $C$ and $B$, and a *-homomorphism $\beta: B \rightarrow B \otimes C$.
Assume that $S \subseteq \mathcal{S}(C)$ is a set of states on $C$ such that the convex hull of $S$ is weakly-*
dense in the state space $\mathcal{S}(C)$.
Then, for any lower semicontinuous seminorm $L$ on $B$, the following are equivalent:
\begin{itemize}
\item[(i)] for every $\omega \in S$ and every $b \in B$, we have $L((\iota \otimes \omega) \beta(b)) \leq L(b)$
\item[(ii)] for every $\omega \in \mathcal{S}(C)$ and every $b \in B$, we have $L((\iota \otimes \omega) \beta(b)) \leq L(b)$.
\end{itemize}
\end{lemma}

\begin{proof}
Suppose {\it (i)} holds.
First we check that the inequality in {\it (ii)} holds if $\omega$ is taken from the convex hull of $S$.
Indeed, suppose $\omega = \sum_{i=1}^k \lambda_i \omega_i$ for some states $\omega_i$ in $S$
and some positive numbers $\lambda_i$ with $\sum_{i=1}^k
\lambda_i = 1$. Then, for any $b \in B$,
\[L((\iota \otimes \omega) \beta(b)) \leq \sum_{i=1}^k \lambda_i L((\iota \otimes \omega_i) \beta(b))
\leq \sum_{i=1}^k \lambda_i L(b) = L(b).\]

Now, if $\omega$ is an arbitrary state on $C$, then $\omega$ is the weak-* limit of some net of states
$\omega_i$ in the convex hull of $S$. This implies that, for any $b \in B$, the net
$(\iota \otimes \omega_i) \beta(b)$ converges weakly to $(\iota \otimes \omega) \beta(b)$.
Hence $(\iota \otimes \omega)\beta(b)$ belongs to the weak closure of the convex set
\[Y = \{(\iota \otimes \psi) \beta(b) \; | \; \psi \mbox{ in the convex hull of } S\}.\]
But the weak closure of a convex set equals its norm closure. Therefore we can take a net of states $\psi_n$ in the convex hull of $S$, such that the net $(\iota \otimes \psi_n) \beta(b)$
norm-converges to $(\iota \otimes \omega) \beta(b)$. Using the lower semi-continuity of $L$, we find
\[L((\iota \otimes \omega) \beta(b)) = L(\lim_i(\iota \otimes \psi_i) \beta(b))
\leq \lim_i L((\iota \otimes \psi_i) \beta(b))
\leq L(b). \]
\end{proof}

There are two half classical cases that have been studied before. We prove that in these settings, the defined notions of 1-isometries and full-isometries both coincide with the existing notion of isometries. First we study the case where $A$ is commutative, and thus of the form $C(G)$ for a compact group $G$.

\begin{theorem} \label{theorem equivalence for classical group acting on CQMS}
Let $G$ be a compact group, $(B,L)$ a CQMS and $\alpha: B \to B \otimes C(G)$ a coaction. Then the following are equivalent:
\begin{enumerate}[(i)]
\item \label{item classical group on CQMS: classical notion} $\alpha$ is isometric in the sense of $(\ref{definition isometric action of CG on CQMS})$
\item \label{item classical group on CQMS: 1-isometric} $\alpha$ is 1-isometric (Definition \ref{definition 1-isometric *-homomorphism})
\item \label{item classical group on CQMS: full-isometric} $\alpha$ is full-isometric (Definition \ref{definition full-isometric coaction}).
\end{enumerate}
\end{theorem}

\begin{proof}
\begin{itemize}
\item (\ref{item classical group on CQMS: full-isometric}) implies (\ref{item classical group on CQMS: 1-isometric}): straightforward.
\item (\ref{item classical group on CQMS: 1-isometric}) implies (\ref{item classical group on CQMS: classical notion}). For $s \in G$, we write $\omega_s$ for the state in $C(G)$ defined in (\ref{definition states omega_s}). Then we use $\alpha_s$ to denote the mapping
    $$\alpha_s: B \to B: b \mapsto \alpha(b)(s) = (\iota \otimes \omega_s) \alpha(b).$$
    We fix elements $b \in B$ and $s \in G$. By (\ref{item classical group on CQMS: 1-isometric}) we know that $L(\alpha_s(b)) \leq L(b)$. On the other hand we also have
    $$L(b) = L(\alpha_e(b)) = L(\alpha_{s^{-1}}(\alpha_s(b))) \leq L(\alpha_s(b)).$$
    Both inequalities together prove (\ref{item classical group on CQMS: classical notion}).
\item (\ref{item classical group on CQMS: classical notion}) implies (\ref{item classical group on CQMS: full-isometric}). Choose faithful QTGs $(C_i, \Delta_i, \beta_i)$ for $i=1, \cdots, n$ such that $\beta = \beta_1 \ast \cdots \ast \beta_n$ is 1-isometric. We will write $C$ for $C_1 \otimes \cdots \otimes C_n$. To prove that $\alpha \ast \beta$ is 1-isometric, we fix an element $b \in B$ and a pure state $\omega$ on $C(G) \otimes C$. Since $C(G)$ is commutative, we can use theorem 4.14 of \cite{Takesaki-Theory_of_operator_algebras_I} to write $\omega = \omega_s \otimes \varphi$ for an element $s \in G$ and a pure state $\varphi$ on $C$. But then, by (\ref{item classical group on CQMS: classical notion}), we have
    $$L((\iota \otimes \omega) (\alpha \ast \beta)(b)) = L(\alpha_s((\iota \otimes \varphi) \beta(b))) = L((\iota \otimes \varphi) \beta(b)) \leq L(b).$$
    Lemma \ref{theorem inequality 1-isometric sufficient for weak*-dense convex set of states} shows that this inequality holds for all states on $C(G) \otimes C$ since it holds for all pure states.
\qedhere
\end{itemize}
\end{proof}

More specifically, this theorem implies that in the double-classical case, where $A = C(G)$ for a compact group $G$ and $B=C(X)$ for a metric space $(X,d)$, the definitions \ref{definition 1-isometric *-homomorphism} and \ref{definition full-isometric coaction} are equivalent to the usual notion of an isometric group action.

There is another special case of our setting that has been studied before, namely the case of a compact quantum group
acting on a finite classical metric space. This has been studied by Banica in \cite{Banica-Quantum_automorphism_groups_of_small_metric_spaces}. The setting is as follows.

Let $\alpha: C(X) \rightarrow C(X) \otimes A$ be a coaction of a CQG $(A, \Delta)$ on a finite metric space
$(X,d)$ with $n$ points.  A coaction on any finite space $X$ with $n$ points is completely described by an
$(n\times n)$ matrix $a=\left( a_{ij} \right)_{i,j=1,\dots,n}$
over $A$. Indeed, consider the standard basis $\{\delta_1, \dots, \delta_n\}$ of $C(X)$,
where $\delta_i$ is the function that is one in the $i$-th point of $X$, and zero elsewhere.
The coaction $\alpha$ is determined by the elements $\alpha(\delta_j)$ for $j \in \{1, \dots, n\}$. As the element $\alpha(\delta_j)$ belongs to $C(X) \otimes A$ we can write
\[\alpha(\delta_j) = \sum_{i=1}^n \delta_i \otimes a_{ij}\]
for some elements $a_{ij}$ in  $A$.
The properties of $\alpha$ being a coaction translate into properties of the elements $a_{ij}$ (see \cite{Wang-Quantum_symmetry_groups_of_finite_spaces}):
the matrix $a=(a_{ij})$
has to be a so called \textit{magic biunitary}, which means that its rows and columns are partitions of the unity of $A$ with projections,
or, explicitly, for all $i,j \in \{1, \cdots, n\}$, we have:
\begin{equation}\label{equation properties magic biunitary}
a_{ij}^* = a_{ij} = a_{ij}^2, \ \
\sum_{k=1}^n a_{ik} = 1, \ \
\sum_{k=1}^n a_{kj} = 1.
\end{equation}
The comultiplication on the elements $a_{ij}$ is given by
$$\Delta(a_{ij})=\sum_{k=1}^n a_{ik}\otimes a_{kj}.$$
Note that (\ref{equation properties magic biunitary}) implies that
\[a_{ij}a_{ik} = 0 = a_{ji}a_{ki}\]
for all $i,j,k \in \{1, \cdots, n\}$ with $j \neq k$.
Moreover, if the coaction $\alpha$ is faithful, i.e.~if $A$ is generated by $\{a_{ij}\mid i,j=1,\dots,n\}$,
then $(A,\Delta)$ is of Kac type with bounded antipode $S$ given by $S(a_{ij})=a_{ji}$.

Now suppose $d$ is a metric on $X$. Consider the $(n\times n)$-matrix $\left( d(i,j) \right)_{i,j=1,\dots,n}$ which we
also denote by $d$. In  \cite{Banica-Quantum_automorphism_groups_of_small_metric_spaces} Banica calls the coaction $\alpha$ of $A$ on $X$ isometric if
the matrices $a$ and $d$ commute:
\begin{equation} \label{definition isometric action of CQG on finite metric space}
a d = d a.
\end{equation}

Does Banica's notion of isometric coactions coincide with the notions introduced
in definitions \ref{definition 1-isometric *-homomorphism} and \ref{definition full-isometric coaction} for the general `double-quantum' setting?
This question is answered in the following theorem, which we will prove in the next section.

\begin{theorem} \label{theorem equivalence for CQG acting on finite metric space}
Let $(X,d)$ be a finite metric space with $n$ points and $(A, \Delta)$ a CQG acting faithfully
on $X$ by the coaction $\alpha: C(X) \rightarrow C(X) \otimes A$.

Take elements $a_{ij}$ in $A$ such that $\alpha(\delta_j) = \sum_{i=1}^n \delta_i \otimes a_{ij}$
where $\delta_j$ is 1 on the $j$-th point of $X$ and 0 elsewhere. Write
\begin{itemize}
\item $L_d(f) = \sup \left\{ \left. \frac{|f(x)-f(y)|}{d(x,y)} \right| x,y \in X, x \neq y \right\}$ for $f \in C(X)$
\item $a$ for the $(n \times n)$-matrix $\left( a_{ij} \right)_{i,j = 1 \cdots n}$
\item $d$ for the $(n \times n)$-matrix $\left( d(i,j) \right)_{i,j = 1 \cdots n}$.
\end{itemize}
Then the following are equivalent:
\begin{enumerate}[(a)]
\item \label{item CQG on FMS: banicas notion} $a  d = d  a$
\item \label{item CQG on FMS: 1-isometric} $\alpha$ is 1-isometric (Definition \ref{definition 1-isometric *-homomorphism})
\item \label{item CQG on FMS: full-isometric}$\alpha$ is full-isometric (Definition \ref{definition full-isometric coaction}).
\end{enumerate}
\end{theorem}

\section{Proof of Theorem \ref{theorem equivalence for CQG acting on finite metric space}}

\subsection{Part 1: (\ref{item CQG on FMS: 1-isometric}) implies (\ref{item CQG on FMS: banicas notion})}

Before we can start the proof of the  first implication of Theorem \ref{theorem equivalence for CQG acting on finite metric space},
we want to express the commutation of $a$ and $d$ in a different, for our purposes more practical way.
This is done in the following lemma. We will need this lemma in the second part of the proof too.

\begin{lemma} \label{theorem practical expression commutation relation}
Using the notations of Theorem \ref{theorem equivalence for CQG acting on finite metric space}, the following are equivalent:
\begin{enumerate}[(A)]
\item \label{item practical commutation: original} $a d = d  a$
\item \label{item practical commutation: new} $a_{ij}a_{kl} = 0$ for every $i,j,k,l \in X$ with $d(i,k) \neq d(j,l)$.
\end{enumerate}
\end{lemma}

\begin{proof}
 Suppose (\ref{item practical commutation: original}) holds, so $a$ and $d$ commute. We want to prove (\ref{item practical commutation: new}). Take points $i,j,k,l$ in $X$ such that $d(i,k) \neq d(j,l)$. Because $a$ and $d$ commute, we have
    \[\sum_{x=1}^n a_{ix} d(x,l) = \sum_{y=1}^n d(i,y) a_{yl}\]
    If we multiply this equality on the left by $a_{ij}$, we get
    \[a_{ij} d(j,l) = \sum_{y=1}^n d(i,y) a_{ij}a_{yl}\]
    because $a_{ix}a_{iy}$ is zero whenever $x \neq y$. Similarly, by multiplying this by $a_{kl}$ on the right, we get
    \[a_{ij} a_{kl} d(j,l) = d(i,k) a_{ij} a_{kl}.\]
    Because $d(j,l) \neq d(i,k)$ this equality is only possible if $a_{ij}a_{kl}$ is zero which is what we wanted to prove.

Conversely, suppose (\ref{item practical commutation: new}) holds. We want to prove (\ref{item practical commutation: original}), so we need $a$ and $d$ to commute. But for every $i,j$ in $X$, we have
$$ (a  d - d  a)_{ij} = \sum_{x=1}^n a_{ix} d(x,j) - d(i,x)a_{xj} =  \sum_{x,y=1}^n d(x,j) a_{ix} a_{yj} - d(i,y) a_{ix} a_{yj} = 0$$
The last equality expresses the given fact that $a_{ix} a_{yj}$ is zero whenever $d(x,j)$ is different from $d(i,y)$. We also used the fact that the rows and columns of $a$ sum up to the unit element.
\end{proof}

Now we can start proving that (\ref{item CQG on FMS: 1-isometric}) implies (\ref{item CQG on FMS: banicas notion}) in Theorem \ref{theorem equivalence for CQG acting on finite metric space}.
Using the notations of that theorem, we suppose $L_d((\iota \otimes \omega) \alpha(f)) \leq L_d(f)$
for every $f \in C(X)$ and every state $\omega$ on $A$. We want to prove $a  d = d a$.

Before we give the actual proof, we need another lemma.
The formulation of that lemma needs some notations.
Since we are working in a finite metric space, the set $D = \{d(i,j) \; | \; i,j \in X\}$ of all distances is a finite set.
So we can rearrange the numbers in $D$, writing $D = \{d_0, d_1, \cdots, d_N\}$ where $d_k < d_l$ if $k < l$.
For every distance $d_\gamma \in D$, and every point $i \in X$, we can look at the set of all points that are at a
distance $d_\gamma$ from $i$. We denote this set by $V_i^\gamma = \{j \in X \; | \; d(i,j) = d_\gamma\}$.

\begin{lemma} \label{theorem main lemma for equivalence CQG on FMS: 1-isometric implies commutation}
We use the notations of Theorem \ref{theorem equivalence for CQG acting on finite metric space} and assume the inequality in (\ref{item CQG on FMS: 1-isometric}).
If a state $\omega$ on $A$ is one on an element $a_{ij}$, then it is zero on every $a_{kl}$ with $d(i,k) \neq d(j,l)$.
In other words: for all $\gamma \in \{0, \cdots, N\}$, all $i,j \in X$ and all $\omega \in \mathcal{S}(A)$ with $\omega(a_{ij}) = 1$, we have
\[\begin{array}{c}
\forall k \in V_i^\gamma, \forall l \not\in V_j^\gamma: \omega(a_{kl})=0, \\
\forall k \not\in V_i^\gamma, \forall l \in V_j^\gamma: \omega(a_{kl})=0.
\end{array} \]
\end{lemma}

\begin{proof}

We prove this lemma by induction on $\gamma$.

\underline{Step 1:} First of all, we want to check the lemma for $\gamma = 0$.
We fix points $i,j \in X$ and a state $\omega$ on $A$ such that $\omega(a_{ij})=1$.
The value $\gamma=0$ corresponds with the smallest distance in the metric space, which is zero.

First, we take $k \in V_i^0$. This means that $d(i,k) = d_0 = 0$, so $k=i$.
We also take some $l \not\in V_j^0$. (hence $l \neq j$) and we want to show that $\omega(a_{kl})$ is zero.
But $\omega(a_{kl}) = \omega(a_{il}) \leq \sum_{x \neq j} \omega(a_{ix})$,
because all $a_{ix}$ are positive and $\omega$ is a positive map.
Moreover,  $\sum_{x\in X}a_{ix}=1$ and since $\omega(1) = 1$,
we have $\omega(a_{kl}) \leq 1 - \omega(a_{ij}) = 0$. Hence $\omega(a_{kl})=0$,
which is the first item we had to prove.

Secondly, we take $k \not\in V_i^0$, so $k \neq i$, and $l \in V_j^0$, so $l=j$.
Then we have, by a similar argument as before, that
$\omega(a_{kl}) = \omega(a_{kj}) \leq \sum_{x \neq i} \omega(a_{xj}) = 1 - \omega(a_{ij}) = 0$, so $\omega(a_{kl})=0$.
\bigskip

\underline{Step 2:} We proceed by induction.
We suppose that the lemma holds for every value in $\{0, \dots, \gamma\}$ for a certain $\gamma$.
We want to prove that it also holds for $\gamma + 1$. Again, we fix points $i$ and $j$ in $X$,
and a state $\omega$ on $A$ such that $\omega(a_{ij})=1$. We take $k \in V_i^{\gamma + 1}$,
and $l \not\in V_j^{\gamma + 1}$.

We are going to use the inequality
\begin{equation} \label{inequality 1-isometric for classical metric space}
L_d((\iota \otimes \omega) \alpha(f)) \leq L_d(f)
\end{equation}
for the chosen state $\omega$ and for $f=D_j$, where $D_j: X \to \mathbb{C}: x \mapsto d(x,j)$.
Using the triangle inequality, it is easy to check that $L_d(D_j)=1$.
So inequality (\ref{inequality 1-isometric for classical metric space}) implies
\[\frac{|(\iota \otimes \omega) \alpha(D_j)(k) - (\iota \otimes \omega) \alpha(D_j)(i)|}{d(k,i)} \leq 1.\]
Notice that $D_j=\sum_{x\in X}d(j,x)\delta_x$, so, using the formula from Theorem \ref{theorem equivalence for CQG acting on finite metric space} for the $\alpha(\delta_x)$, we get
\[\left| \sum_{x \in X} d(j,x) \omega(a_{kx}) - \sum_{x \in X} d(j,x) \omega(a_{ix}) \right| \leq d(i,k) = d_{\gamma + 1}\]
In the first step, we have proven that $\omega(a_{ix})=0$ whenever $x \neq j$. We also know that $d(j,j)=0$.
Those two facts give
\begin{equation} \label{inequality reduced to convex combination}
\left| \sum_{x \in X} d(j,x) \omega(a_{kx}) \right| \leq d_{\gamma + 1}.
\end{equation}
Because of the induction hypothesis, we have that $\omega(a_{kx})=0$ if $d(j,x) < d_{\gamma + 1}$.
Indeed, in that case, we can find a value $c \in \{0, \cdots, \gamma\}$ such that $d(j,x)=d_c$ and hence $x \in V_j^c$.
On the other hand, since $k \in V_i^{\gamma + 1}$, we know that $k \not\in V_i^c$. As $c\le\gamma$, it follows from
the induction hypothesis that $\omega(a_{kx})=0$.
This means that the left hand side of inequality (\ref{inequality reduced to convex combination})
is a convex combination of those distances $d(j,x)$ that are at least $d_{\gamma + 1}$,
although the convex combination itself is smaller than $d_{\gamma + 1}$.
It is clear that is only possible if the co\"effici\"ents of the distances $d(j,x) > d_{\gamma + 1}$ are zero.
So, $\omega(a_{kx})=0$ unless $d(j,x) = d_{\gamma + 1}$. In particular,
since $l \not\in V_j^{\gamma + 1}$, this implies that $\omega(a_{kl}) = 0$.

We still have to prove the second case, for $k \not\in V_i^{\gamma + 1}$ and $l \in V_j^{\gamma + 1}$.
In this case, we define a new state
\[\omega': A \rightarrow \mathbb{C}: a \mapsto \omega(S(a))\]
where $S$ is the antipode of $(A,\Delta)$. Since $\omega'(a_{ji}) = \omega (S(a_{ji}))=\omega(a_{ij})=1$,
we can use the first case of this lemma on the distance $\gamma + 1$, the points $j$ and $i$ in $X$ and the state $\omega'$.
Since $l \in V_j^{\gamma + 1}$ and $k \not\in V_i^{\gamma + 1}$, the first case states that $0=\omega'(a_{lk})=\omega(a_{kl})$
and this concludes the proof.
\end{proof}

Using the previous lemma, we now need to show the commutation of $a$ and $d$ to complete
the proof of the first implication of Theorem \ref{theorem equivalence for CQG acting on finite metric space}.
Choose any  $i,j,k,l \in X$ with $d(i,k) \neq d(j,l)$. By Lemma \ref{theorem practical expression commutation relation} it suffices to show that
$a_{ij}a_{kl}=0$. Suppose that the product $a_{ij}a_{kl}$ is nonzero. Then also $a_{ij}a_{kl}a_{ij} = (a_{ij}a_{kl})(a_{ij}a_{kl})^*$
is nonzero. Hence we can find a state $\omega$ on $A$ such that $\omega(a_{ij}a_{kl}a_{ij})$ is nonzero.
But then also $\omega(a_{ij})$ is nonzero, because of the Cauchy-Schwarz inequality:
\[0 < \omega(a_{ij}(a_{kl}a_{ij}))^2 \leq \omega(a_{ij}) \omega(a_{ij}a_{kl}a_{ij}).\]
Using this state $\omega$, we define a new state
\[\omega_{ij}: A \rightarrow \mathbb{C}: x \mapsto \frac{\omega(a_{ij}xa_{ij})}{\omega(a_{ij})}.\]
One can easily check that this is indeed a state and that $\omega_{ij}(a_{ij})=1$.
So we can use the first case of the Lemma \ref{theorem main lemma for equivalence CQG on FMS: 1-isometric implies commutation} on the number $\gamma$
for which $d(i,k)=d_\gamma$, on the points $i$ and $j$, and on the state $\omega_{ij}$, and conclude that
$\omega_{ij}(a_{kl})=0$. But this clearly contradicts the fact that
$\omega_{ij}(a_{kl}) = \frac{\omega(a_{ij}a_{kl}a_{ij})}{\omega(a_{ij})}$
is nonzero because of the choice of $\omega$. Therefore $a_{ij}a_{kl}=0$, and this concludes the proof of this first implication.

\subsection{Part 2: (\ref{item CQG on FMS: banicas notion}) implies (\ref{item CQG on FMS: 1-isometric})}

Since this part of the theorem will not use anything specific about the CQG-structure we have on $A$ (except for the fact that the matrix $a$ is a magic biunitary), we can reformulate this part in a more general setting.

\begin{theorem} \label{theorem equivalence CQG on FMS: commutation implies 1-isometric for magic biunitary}
Let $(X,d)$ be a finite metric space with $n$ points, $A$ a unital C*-algebra and $\alpha: C(X) \rightarrow C(X) \otimes A$ a *-homomorphism.

Take elements $a_{ij}$ in $A$ such that $\alpha(\delta_j) = \sum_{i=1}^n \delta_i \otimes a_{ij}$. We use the same notations as in Theorem \ref{theorem equivalence for CQG acting on finite metric space} for $\delta_i$, $L_d$, $a$ and $d$. Suppose $a$ is a magic biunitary matrix. Then, if $ad=da$, it follows that $\alpha$ is 1-isometric.
\end{theorem}

We use the notations from the theorem, and we suppose that the magic biunitary matrix $a$ commutes with $d$. We want to prove that $L_d((\iota \otimes \omega) \alpha(f)) \leq L_d(f)$ for every $f \in C(X)$ and every state $\omega$ on $A$.

First note that it is sufficient to prove the following lemma:

\renewcommand\arraystretch{1.3}
\begin{lemma} \label{theorem use of combinatorial lemma for CQG on FMS}
 Using the notations and assumptions of Theorem \ref{theorem equivalence CQG on FMS: commutation implies 1-isometric for magic biunitary}, we fix elements $x,y \in X$ and a state $\omega$ on $A$. Then there exist positive numbers $\lambda_{ij} \; (i,j \in X)$ such that for all $i,j \in X$, we have
\[\left\{ \begin{array}{l}
\sum_{j \in X} \lambda_{ij} = \omega(a_{xi}) \\
\sum_{i \in X} \lambda_{ij} = \omega(a_{yj}) \\
\lambda_{ij} = 0 \mbox{ if } d(i,j) \neq d(x,y)
\end{array} \right.\]
\end{lemma}
\renewcommand\arraystretch{2}

Suppose the lemma holds, then we can prove Theorem \ref{theorem equivalence CQG on FMS: commutation implies 1-isometric for magic biunitary}.
Indeed, we fix a state $\omega \in \mathcal{S}(A)$ and a function $f \in C(X)$. Now we can rewrite
\[(\iota \otimes \omega) \alpha(f) =
\sum_{i \in X} f(i) (\iota \otimes \omega) \alpha(\delta_i) = \sum_{i,j \in X} f(i) \omega(a_{ji}) \delta_j\]

Choose elements $x,y \in X$ with $x\neq y$. Using Lemma~\ref{theorem use of combinatorial lemma for CQG on FMS}, we can write
\begin{eqnarray*}
\frac{\left|(\iota \otimes \omega) \alpha(f) (x) - (\iota \otimes \omega) \alpha(f) (y) \right|}{d(x,y)} & = & \frac{\left| \sum_{i \in X} f(i) (\omega(a_{xi}) - \omega(a_{yi})) \right|}{d(x,y)} \\
& = & \frac{\left| \sum_{i \in X} f(i) \left( \sum_{j \in X} \lambda_{ij} - \sum_{j \in X} \lambda_{ji} \right) \right|}{d(x,y)} \\
& = & \frac{\left| \sum_{i,j \in X} \lambda_{ij}(f(i)-f(j)) \right|}{d(x,y)} \\
& = & \left| \sum_{i,j \in X} \lambda_{ij} \; \frac{f(i)-f(j)}{d(i,j)} \right| \\
& \leq & \sum_{i,j \in X} \lambda_{ij} \; L_d(f) \\
& = & L_d(f)
\end{eqnarray*}

\renewcommand\arraystretch{1}

Because this holds for all elements $x$ and $y$ in $X$ with $x\neq y$, we find that
\[L_d((\iota \otimes \omega) \alpha(f)) \leq L_d(f)\]
for every $f \in C(X)$, which would conclude the proof of Theorem \ref{theorem equivalence CQG on FMS: commutation implies 1-isometric for magic biunitary}.
\medskip

Of course, we still have to prove Lemma \ref{theorem use of combinatorial lemma for CQG on FMS}. This lemma is combinatorial,
and its proof will use a famous theorem in graph theory. To formulate this theorem, we need to introduce some notations.

The graph theory we need concerns flow networks. A flow network is a set $V$ of vertices, and a set $E \subseteq V \times V$
of directed edges. Every edge $(u,v)$ in $E$ has a positive capacity $c(u,v)$. We also fix two vertices $s$ and $t$,
called the source and the sink.

An $s$-$t$-flow in such a network is a mapping from the set of edges to the positive numbers,
which maps every edge $(v,w) \in E$ on $f(v,w)$, the 'flow from $v$ to $w$'. The conditions are
that the flow $f(v,w)$ through an edge is always smaller than the capacity $c(v,w)$ of the edge.
Secondly, the flow can not stay inside a vertex (except for the source and the sink).
So the flow entering a vertex must equal the flow leaving the vertex, or
$\sum_{u \in V} f(u,v) = \sum_{w \in V} f(v,w)$ for every vertex $v \in V \backslash \{s,t\}$.
In this mathematical notation we suppose that the capacity $c(v,w)$ (and hence also the flow $f(v,w)$)
is zero when there is no edge $(v,w)$ in $E$.

The value $F$ of an $s$-$t$-flow is then the sum of all the flow leaving the source $\sum_{v \in V} f(s,v)$.
This equals the flow entering the sink.

An $s$-$t$-cut is a partition of the vertices. We cut the graph into two pieces, partitioning the set of vertices
in two sets, $S$ and $T$ in such a way that $S$ contains the source $s$ and $T$ contains the sink $t$.
The capacity of such a cut is then the sum of the capacities of all edges crossing the cut from $S$ to $T$.
Mathematically, this is $\sum_{v \in S, w \in T} c(v,w)$.

Now we can formulate the famous max-flow-min-cut theorem briefly.

\begin{theorem}[Max-flow-min-cut theorem]
The maximal value of an $s$-$t$-flow is equal to the minimal capacity of an $s$-$t$-cut.
\end{theorem}

This theorem will be used to prove the following combinatorial lemma.

\begin{lemma} \label{theorem combinatorial lemma}
Suppose we have positive numbers $\alpha_1, \alpha_2, \cdots, \alpha_n$ and \\
$\beta_1, \beta_2, \cdots, \beta_n$ such that $\sum_{i=1}^n \alpha_i = \sum_{i=1}^n \beta_i = 1$.
Next, suppose that for every $i \in \{1, \cdots, n\}$, there exists a set
$V_i \subseteq \{1, \cdots, n\}$ such that for every subset $Z$ of $\{1, \cdots, n\}$
\begin{equation} \label{inequality in combinatorial lemma} \sum_{i \in Z} \alpha_i \leq \sum_{j \in \bigcup_{i \in Z} V_i} \beta_j.\end{equation}

\renewcommand\arraystretch{1.3}

Then, for $i,j \in \{1, \cdots, n\}$, we can take positive numbers $\lambda_{ij}$ such that for all $i,j \in \{ 1, \cdots, n\}$
\[ \left\{ \begin{array}{l}
\sum_{j=1}^n \lambda_{ij} = \alpha_i \\
\sum_{i=1}^n \lambda_{ij} = \beta_j \\
\lambda_{ij} = 0 \mbox{ if } j \not\in V_i
\end{array} \right.\]
\end{lemma}

\renewcommand\arraystretch{1}

\begin{proof}
We construct a flow network. We have $n$ vertices on the left hand side, which we label $l_1, l_2, \cdots, l_n$,
and $n$ vertices on the right hand side, labeled $r_1, r_2, \cdots, r_n$.
For short, we denote $L$ for $\{l_1, \cdots, l_n\}$ and $R$ for $\{r_1, \cdots,  r_n\}$.
For every $i \in \{1, \cdots, n\}$, we have an edge from the source $s$ to vertex $l_i$ with capacity
$\alpha_i$, and an edge from vertex $r_i$ to the sink $t$ with capacity $\beta_i$.
In the center, we have an edge from  vertex $l_i$ to every vertex $r_j$ with $j \in V_i$.
We give those edges capacity 1. We denote the set of vertices by $V$ and the set of edges by $E$.
From now on, we will see $V_i$ as the subset of all vertices $r_j$ with $j \in V_i$.

An example for $n=5$ can be found on figure \ref{figure example flow network}.
\begin{figure}[H]
\centering
\includegraphics[scale=.22, angle=90]{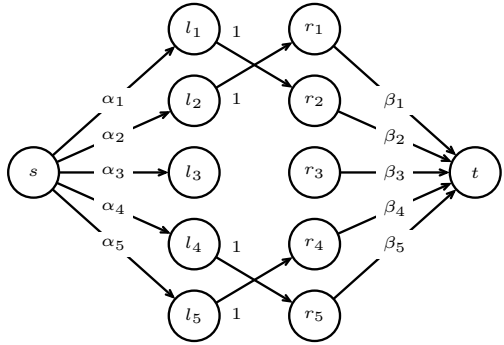}
\caption{Example of a flow network when $n=5$} \label{figure example flow network}
\center{$V_1=\{2\}, V_2=\{1\}, V_3=\emptyset, V_4=\{5\}, V_5=\{4\}$}
\end{figure}

We want to use the max-flow-min-cut theorem to find a maximal flow in this flow network. For this end, we investigate the minimal capacity of an $s$-$t$-cut. An example of an $s$-$t$-cut with capacity one is $S = \{s\}, T = V \backslash \{s\}$. We claim that this capacity is minimal. So we want to prove that every other $s$-$t$-cut has a capacity of at least one. Let us consider all possible cases.

\begin{itemize}

\item First we suppose $S$ contains non of the vertices on the left hand side: $L \subseteq T$. Then the capacity of the cut is
\[\sum_{\stackrel{v \in S}{w \in T}} c(v,w) \geq \sum_{w \in L} c(s,w) = \sum_{i=1}^n c(s,l_i) = \sum_{i=1}^n \alpha_i = 1.\]

\item If one of the central edges is crossing the cut from $S$ to $T$, then the capacity of the cut is bigger than one.
This is, if $S \cap L \neq \emptyset$ and $T \cap \bigcup_{l_i \in S} V_i \neq \emptyset$,
then we have a vertex $r_j \in T$ in one of the $V_i$, with $l_i \in S$.
Then $l_i$ belongs to $S$ and $r_j$ to $T$, so the capacity of the cut is bigger than $c(l_i, r_j)$,
which is 1 because $r_j \in V_i$.

\item The last case is the case where $S \cap L \neq \emptyset$ but $T \cap \bigcup_{l_i \in S} V_i = \emptyset$.
This means that $V_i \subseteq S$ when $l_i \in S$. If we now use the given inequality (\ref{inequality in combinatorial lemma}) on the set $S \cap L$, we get
\[\sum_{l_i \in S} \alpha_i \leq \sum_{r_j \in \bigcup_{l_i \in S} V_i} \beta_j\]
For the total capacity of the cut, this gives
\begin{multline*}
\sum_{\stackrel{v \in S}{w \in T}} c(v,w) \geq \sum_{w \in (T \cap L)} c(s,w) + \sum_{v \in (S \cap R)} c(v,t) \\
= \sum_{l_i \in T} \alpha_i + \sum_{r_j \in \bigcup_{l_i \in S} V_i} \beta_j \geq \sum_{l_i \in T} \alpha_i + \sum_{l_i \in S} \alpha_i = \sum_{i=1}^n \alpha_i = 1
\end{multline*}

\end{itemize}

We have proven that the minimal capacity of an $s$-$t$-cut in the flow network we consider is one.
Because of the max-flow-min-cut theorem, this implies that we can have a maximal flow $f$ of value one through this network.
If we denote the flow in the edge $(l_i, r_j)$ by $\lambda_{ij}$, we have found the positive numbers we were looking for:

\begin{itemize}

\item It is obvious that all $\lambda_{ij} = f(l_i, r_j)$ are positive.

\item For the first condition, we have to calculate $\sum_{j=1}^n \lambda_{ij} = \sum_{j=1}^n f(l_i, r_j)$.
This is the total amount of flow leaving the vertex $l_i$, so this has to equal the amount of flow entering $l_i$.
The only edge through which the flow can enter $l_i$, is the edge $(s,l_i)$, which has capacity $\alpha_i$.
In order to have a total flow of value one, the flow in every edge $(s,l_k)$ must reach its capacity $\alpha_k$ (since $\sum_{k=1}^n \alpha_k = 1$). Hence, in particular, the flow through the edge $(s,l_i)$ equals its capacity $\alpha_i$.
So we have proven that $\sum_{j=1}^n \lambda_{ij} = f(s, l_i) = \alpha_i$.

\item For the second condition, we calculate $\sum_{i=1}^n \lambda_{ij} = \sum_{i=1}^n f(l_i,r_j)$ similarly,
we see that this is the total amount of flow entering the vertex $r_j$. This is equal to the amount of flow leaving $r_j$, which is the flow through the edge $(r_j,t)$.
With a similar reasoning as for the first condition, we see that the flow $f(r_j,t)$ equals the capacity $\beta_j$ of the edge.

\item If $j \not\in V_i$, there is no edge from $l_i$ to $r_j$. Clearly, there cannot be any flow going from $l_i$ to $r_j$, so $\lambda_{ij} = f(l_i,r_j)$ is zero.
\qedhere
\end{itemize}
\end{proof}

We want to use this purely combinatorial lemma to prove Lemma \ref{theorem use of combinatorial lemma for CQG on FMS}.

\begin{proof}[Proof of Lemma \ref{theorem use of combinatorial lemma for CQG on FMS}:]
Fix elements $x$ and $y$ in $X$ and a state $\omega$ on $A$.
It is obvious that we want to use Lemma \ref{theorem combinatorial lemma} for the numbers
$\alpha_i = \omega(a_{xi})$ and $\beta_i = \omega(a_{yi})$.
To make the third condition of Lemma \ref{theorem combinatorial lemma} match with the third condition of
Lemma \ref{theorem use of combinatorial lemma for CQG on FMS}, we take $V_i$ to be the set $\{j \in X \; | \; d(i,j) = d(x,y)\}$.

For example, if the metric space looks like figure \ref{figure example metric space} (where the numbers represent
the distances between the points), and we chose $x$ to be $c_1$ and $y$ to be $c_2$,
then the flow network used in the proof of Lemma \ref{theorem combinatorial lemma} will look like the network in
figure \ref{figure flow network for given metric space} (which we also used as an example on figure \ref{figure example flow network}
in the proof of Lemma \ref{theorem combinatorial lemma}).

\begin{figure}[H]
\centering
\subfigure[]{\label{figure example metric space}
\includegraphics[clip=true, trim= 0 250 0 0, scale=.22,  angle=90]{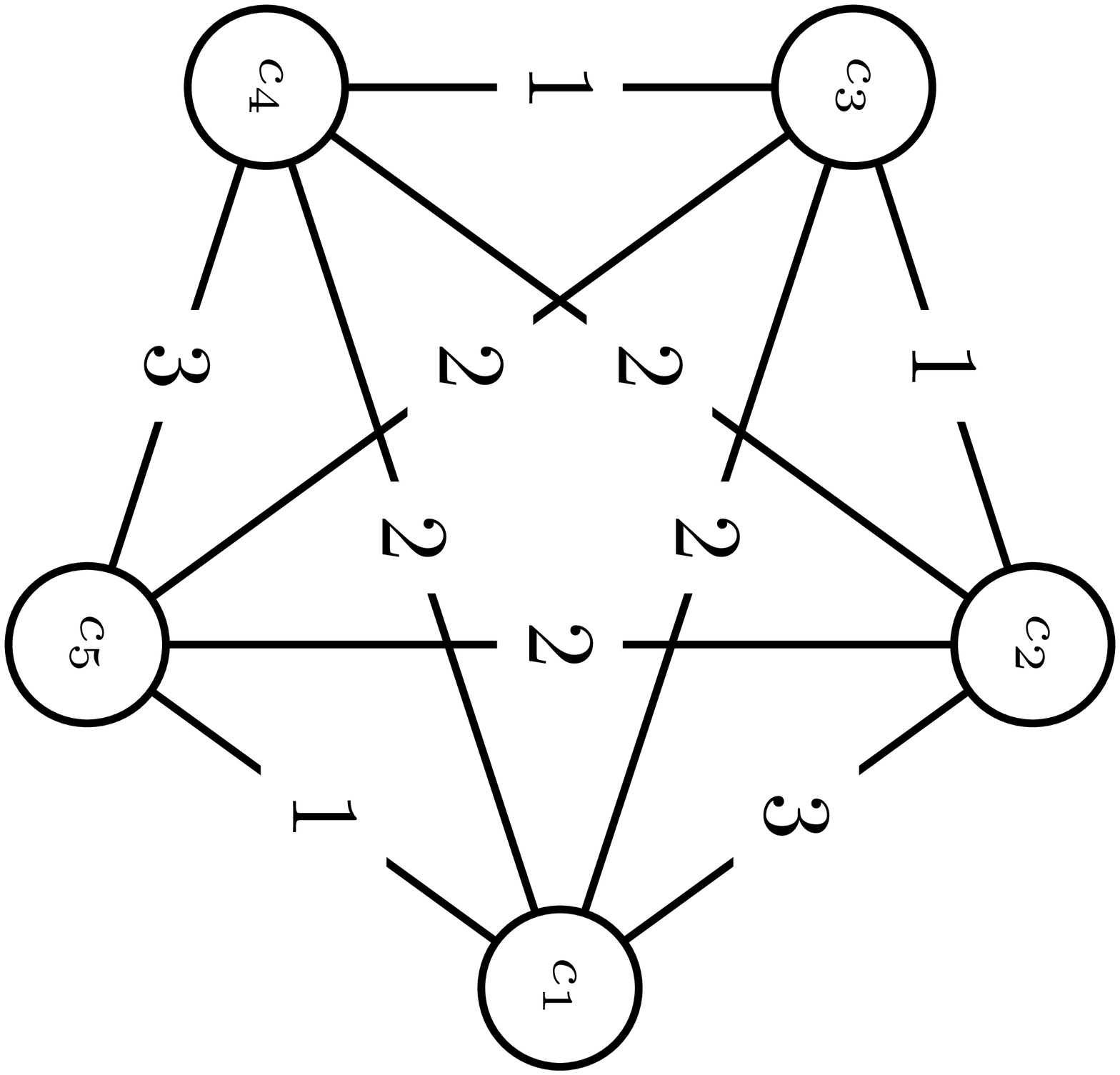}
}
\subfigure[]{\label{figure flow network for given metric space}
\includegraphics[scale=.22, angle=90]{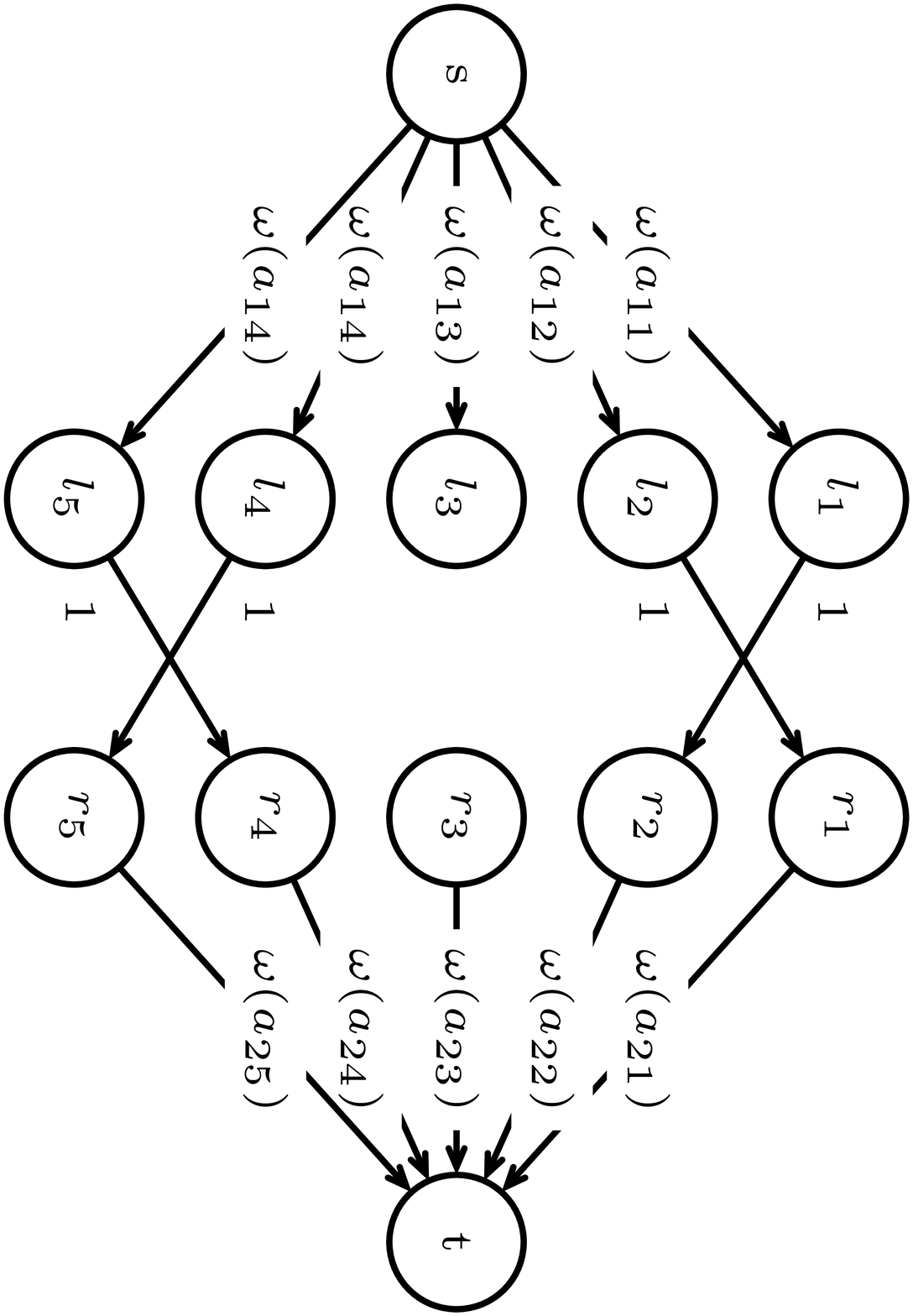}
}
\caption{An example of a metric space and the corresponding flow network}
\end{figure}

In order to use Lemma \ref{theorem combinatorial lemma}, we need to check the inequality
\[\sum_{i \in Z} \omega(a_{xi}) \ \ \leq \sum_{\stackrel{j\in X:}{\exists i \in Z: \; d(i,j) = d(x,y)}} \omega(a_{yj})\]
for every subset $Z \subseteq X$. But this follows from the fact that $a$ and $d$ commute. Indeed,
by using the commutation relation in the way we obtained in Lemma \ref{theorem practical expression commutation relation},
we get for $Z \subseteq X$
\[\sum_{i \in Z} a_{xi} = \sum_{i \in Z} a_{xi} \sum_{j \in X} a_{yj} =
\sum_{i \in Z} a_{xi} \sum_{\stackrel{j \in X:}{\exists i \in Z: \; d(i,j) = d(x,y)}} a_{yj}\]
because $a_{xi} \, a_{yj} = 0$ whenever $d(i,j) \neq d(x,y)$.

Because both factors in the previous product are projections, we have
\[\sum_{i \in Z} a_{xi}\ \ \leq \sum_{\stackrel{j\in X:}{\exists i \in Z: \; d(i,j) = d(x,y)}} a_{yj}.\]
Since $\omega$ is a positive map, we immediately get the desired inequality. Hence we can apply Lemma \ref{theorem combinatorial lemma} to obtain the desired positive numbers $\lambda_{ij}$.
\end{proof}

\subsection{Part 3: Equivalence of (\ref{item CQG on FMS: 1-isometric}) with (\ref{item CQG on FMS: full-isometric})}

Theorem \ref{theorem equivalence for CQG acting on finite metric space} also states that, in case of a finite classical space, the notions 1-isometric (Definition \ref{definition 1-isometric *-homomorphism}) and full-isometric (Definition \ref{definition full-isometric coaction}) coincide. We already know that (\ref{item CQG on FMS: full-isometric}) implies (\ref{item CQG on FMS: 1-isometric}). To prove that (\ref{item CQG on FMS: 1-isometric}) implies (\ref{item CQG on FMS: full-isometric}), we need another lemma.

\begin{lemma} \label{theorem *-product 1-isometric implies individual coactions 1-isometric}
Let $(B,L)$ be a CQMS and suppose we have CQGs $(C_i, \Delta_i)$ with bounded counits $\epsilon_i$ and coactions $\beta_i: B \to B \otimes C_i$ ($i=1, \cdots, m$). If $$\beta := \beta_1 \ast \cdots \ast \beta_m: B \to B \otimes C_1 \otimes \cdots \otimes C_m$$ is 1-isometric, then $\beta_i$ is 1-isometric for every $i \in \{1, \cdots, m\}$.
\end{lemma}
\begin{proof}
We fix $i \in \{1, \cdots, m\}$. Since $(\iota \otimes \epsilon_j) \beta_j = id_{B}$ for every $j \in \{1, \cdots, m\}$, one can compute that
$$(\iota \otimes \epsilon_1 \otimes \cdots \otimes \epsilon_{i-1} \otimes \iota \otimes \epsilon_{i+1} \otimes \cdots \otimes \epsilon_n) \beta = \beta_i.$$

Now we can easily prove that $\beta_i$ is 1-isometric if $\beta$ is 1-isometric: we first choose $b \in B$ and $\omega \in \mathcal{S}(C_i)$. Then, because of the previous formula for $\beta_i$, we have
$$L((\iota \otimes \omega) \beta_i(b)) = L((\iota \otimes \epsilon_1 \otimes \cdots \otimes \epsilon_{i-1} \otimes \omega \otimes \epsilon_{i+1} \otimes \cdots \otimes \epsilon_m) \beta(b)) \leq L(b)$$
since $\epsilon_1 \otimes \cdots \otimes \epsilon_{i-1} \otimes \omega \otimes \epsilon_{i+1} \otimes \cdots \otimes \epsilon_n$ is a state on $C_1 \otimes \cdots \otimes C_m$.
\end{proof}

\begin{proof}[Proof of $(\ref{item CQG on FMS: 1-isometric})$ implies $(\ref{item CQG on FMS: full-isometric})$ in Theorem \ref{theorem equivalence for CQG acting on finite metric space}]
We suppose that the coaction $\alpha$ of the CQG $(A, \Delta)$ on $(C(X), L_d)$ is 1-isometric, which means that the matrix $a = (a_{ij})$ commutes with the distance matrix. To prove that $\alpha$ is full-isometric, we first choose faithful QTGs $(C_i, \Delta_i, \beta_i)$ for $i=1, \cdots, m$ such that $\beta = \beta_1 \ast \cdots \ast \beta_m$ is 1-isometric. We want to prove that $\alpha \ast \beta$ is 1-isometric. We denote $X$ for $A \otimes C_1 \otimes \cdots \otimes C_n$, and we use $x_{ij}$ for the matrix elements in $X$ such that, for every $j \in \{1, \cdots, n\}$ we have
$$(\alpha \ast \beta) (\delta_j) = \sum_{k=1}^n \delta_i \otimes x_{ij}.$$
Then, by Theorem \ref{theorem equivalence CQG on FMS: commutation implies 1-isometric for magic biunitary}, it is sufficient to prove that the matrix $x = (x_{ij})_{ij}$ is a magic biunitary that commutes with the distance-matrix $d$.

We need to compute the elements $x_{ij}$. For every $k \in \{1, \cdots, m\}$ we can take elements $c_{ij}^{(k)}$ in $C_k$ such that $\beta_k(\delta_j) = \sum_{i=1}^n \delta_i \otimes c_{ij}^{(k)}$ for every $j \in \{1, \cdots, n\}$. We can also take elements $a_{ij}$ in $A$ such that for every $j \in \{1, \cdots, n\}$ we have $\alpha(\delta_j) = \sum_{i=1}^n \delta_i \otimes a_{ij}$. Then one can check that for every $i,j \in \{1, \cdots, n\}$, we have
\begin{equation} \label{equation matrix elements x_ij of *-product of coactions} x_{ij} = \sum_{k_1, \cdots, k_m =1}^n a_{ik_1} \otimes c_{k_1k_2}^{(1)} \otimes c_{k_2k_3}^{(2)} \otimes \cdots \otimes c_{k_{m-1}k_m}^{(m-1)} \otimes c_{k_mj}^{(m)} \end{equation}
Now we can check whether the elements $x_{ij}$ satisfy all the conditions needed to apply Theorem \ref{theorem equivalence CQG on FMS: commutation implies 1-isometric for magic biunitary}.
Firstly, since $\alpha$ and all the $\beta_k$ are coactions, we know that the matrix $a = (a_{ij})_{ij}$ and all matrices $c^{(k)} = \left( c_{ij}^{(k)} \right)_{ij}$ are magic biunitary matrices. From this, one can easily see that also $x$ will be a magic biunitary.

Secondly, since $\alpha$ is given to be 1-isometric, we know that $a$ commutes with $d$. We also supposed $\beta$ to be 1-isometric, which by Lemma \ref{theorem *-product 1-isometric implies individual coactions 1-isometric} implies that every coaction $\beta_k$ is 1-isometric. Notice that the conditions of Lemma \ref{theorem *-product 1-isometric implies individual coactions 1-isometric} are met here since all $C_k$ have bounded counits as they are faithfully acting on $C(X)$ \cite{Wang-Quantum_symmetry_groups_of_finite_spaces}. Hence all matrices $c^{(k)}$ commute with $d$.
Then, using (\ref{equation matrix elements x_ij of *-product of coactions}), it is straightforward to verify that $d$ commutes with $x$ as well.
\end{proof}

\section{Subgroup acting isometrically}

Classically, when a group $G$ is acting on a metric space $(X,d)$, the subset $H$ consisting of those elements $g$ of $G$ for which
the action $X\to X:x\mapsto gx$ is isometric forms a subgroup of $G$.
We show in Theorem~\ref{theorem largest quantum subgroup acting isometrically} that this remains true in the general quantum context for full-isometric coactions. Before stating the theorem, we need to introduce a definition.

\begin{definition}
Let $(A, \Delta)$ be a compact quantum group and $I \subseteq A$ a subset of $A$.
We say that $I$ is a \textbf{Woronowicz Hopf C*-ideal} if $I$ is a two-sided closed *-ideal in $A$ such that
\[(p\otimes p) \Delta(I) = \{0\}\]
where $p$ is the canonical projection $A\to A/I:a\mapsto a+I$.
\end{definition}

In the above situation, $A/I$ becomes a compact quantum group for the comultiplication defined by
\[\Delta_I: A/I \rightarrow A/I \otimes A/I: p(a) \mapsto (p \otimes p) \Delta(a).\]
The CQG $(A/I, \Delta_I)$ can be considered as a compact quantum subgroup of $(A,\Delta)$.

\begin{theorem} \label{theorem largest quantum subgroup acting isometrically}
If $\alpha$ is a faithful coaction of a CQG $(A, \Delta)$ with bounded counit $\epsilon$ on a CQMS $(B,L)$, then
there exists a proper Woronowicz Hopf C*-ideal $I$ in $A$ such that
\[\alpha_I: B \rightarrow B \otimes A/I: b \mapsto (\iota \otimes p_I) \alpha(b)\]
is a faithful and full-isometric coaction, where $p_I$ is the canonical projection of $A$ onto $A/I$.
The ideal $I$ can be chosen to be the smallest one possible, in the sense that for any other Woronowicz Hopf C*-ideal $J$ with $\alpha_J = (\iota \otimes p_J) \alpha$ a faithful and full-isometric coaction, we have that $I \subseteq J$.
\end{theorem}

We need a couple of lemmas before we can start the proof of this theorem.

\begin{lemma} \label{theorem separating set with extra condition is only positive on positive elements}
Let $S$ be a set of states on a unital C*-algebra $A$ such that
\renewcommand{\labelenumi}{(\roman{enumi})}
\begin{enumerate}
\item $S$ separates the points of $A$ (i.e. $\forall a \in A \backslash \{0\}, \exists \omega \in S: \omega(a) \neq 0$)
\item For all $a \in A$ and all $\omega \in S$ with $\omega(a^*a) \neq 0$, the state $\omega_a$ belongs to $S$, where $\omega_a$ is defined by
$$\omega_a: A \to \mathbb{C}: x \mapsto \frac{\omega(a^* x a)}{\omega(a^*a)}.$$
\end{enumerate}
\renewcommand{\labelenumi}{(\arabic{enumi})}
Then for every element $a$ of $A$, $a \geq 0$ is equivalent to $\omega(a) \geq 0$ for all $\omega$ in $S$.
\end{lemma}

\begin{proof}
We fix an element $a \in A$. First of all, it is clear that, if $a$ is positive, $\omega(a)$ will also be positive for all $\omega$ in $S$, since every state is a positive mapping. Next, we suppose that $\omega(a) \geq 0$ for all $\omega$ in $S$, and we want to prove that $a$ is positive. First we check that $a$ is self-adjoint. Suppose not, then $a-a^*$ is nonzero. Since the states in $S$ separate the points of $A$, there exists an $\omega \in S$ such that $\omega(a-a^*) \neq 0$. But then $\omega(a) \neq \omega(a^*) = \overline{\omega(a)}$ which contradicts the fact that $\omega(a)$ is a positive (and hence real) number.

So we may suppose that $a$ is self-adjoint. Then we can write $a$ as $a^+-a^-$ for some positive elements $a^+$ and $a^-$ in $A$ with $a^+a^- = 0$. To prove that $a$ is positive, we have to show that $a^-$ is zero. But if $a^-$ were non-zero, then also $(a^-)^3$ would be non-zero. Hence we could find a state $\omega$ in $S$ such that $\omega((a^-)^3) \neq 0$. But then also $\omega((a^-)^2) \neq 0$, since by the Cauchy-Schwarz inequality we have $0 < |\omega((a^-)^3)|^2 \leq \omega((a^-)^2) \omega((a^-)^4)$. By the second property of the set $S$, we now know that $\omega_{a^-}$ belongs to $S$, which implies that $\omega_{a^-}(a) \geq 0$. But
$$\omega_{a^-}(a) = \frac{\omega(a^-aa^-)}{\omega((a^-)^2)} = \frac{\omega(a^-a^+a^- - (a^-)^3)}{\omega((a^-)^2)} = -\frac{\omega((a^-)^3)}{\omega((a^-)^2)}$$
and this would be strictly negative, since both $\omega((a^-)^3)$ and $\omega((a^-)^2)$ are supposed to be positive and non-zero. We get an obvious contradiction and conclude that $a=a^+$ is positive.
\end{proof}

\begin{lemma} \label{theorem set of states only positive on positive elements is convex-weakly *-dense}
Let $S$ be a set of states on a unital C*-algebra $A$ such that for every element $a \in A$ we have $a \geq 0$ iff $\omega(a) \geq 0$ for all $\omega \in S$. Then the convex hull of $S$ is weakly *-dense in the state space $\mathcal{S}(A)$.
\end{lemma}

\begin{proof}
Suppose the lemma does not hold. Then we can take a state $\omega$ in $\mathcal{S}(A) \backslash \overline{\mbox{conv}(S)}$. As $\overline{\mbox{conv}(S)}$ is weakly *-compact, there exists, by the Hahn-Banach separation theorem, a linear, weakly *-continuous mapping $\psi: A^* \to \mathbb{C}$ and a real number $\lambda$ such that
$$\mathfrak{Re}(\psi(\varphi)) \leq \lambda < \mathfrak{Re}(\psi(\omega))$$
for all states $\varphi$ in $\overline{\mbox{conv}(S)}$. But since $\psi$ is linear and weakly *-continuous, there must exist an element $a$ in $A$ such that $\psi(\varphi) = \varphi(a)$ for all $\varphi$ in $A^*$. Now the above equality reads
$$\mathfrak{Re}(\varphi(a)) \leq \lambda < \mathfrak{Re}(\omega(a))$$
for all $\varphi$ in $\overline{\mbox{conv}(S)}$. Since $\omega$ and all $\varphi \in \overline{\mbox{conv}(S)}$ are self-adjoint, we can use $b$ to denote $\mathfrak{Re}(a)$ and then we get
$$\varphi(b) \leq \lambda < \omega(b)$$
for all $\varphi$ in $\overline{\mbox{conv}(S)}$. So, for the element $\lambda 1 - b$ (where 1 is the unit element of $A$), we get $\varphi(\lambda 1-b) = \lambda - \varphi(b) \geq 0$ for all $\varphi \in \overline{\mbox{conv}(S)}$, and hence in particular for all $\varphi$ in $S$. By the assumption in the lemma, this would mean that $\lambda 1 - b$ is a positive element of $A$. But since $\omega$ is a state, this would imply that $\omega(\lambda 1 -b)$ is positive, which gives a contradiction with $\omega(b) > \lambda$.
\end{proof}

\begin{lemma} \label{theorem states composed with separating *-homomorphisms convex-weak *-dense}
Let $A$ be a unital C*-algebra and $\Pi$ be a set of *-homomorphisms $\pi: A \to C_\pi$ where the $C_\pi$ are unital C*-algebras. We suppose $\Pi$ separates the points of $A$, so if $a \in A$ is non-zero, then there exists an element $\pi \in \Pi$ such that $\pi(a) \neq 0$. Then the convex hull of the set
$$S = \{\omega \circ \pi \; | \; \pi \in \Pi, \omega \in \mathcal{S}(C_\pi)\}$$
is weakly *-dense in $\mathcal{S}(A)$.
\end{lemma}

\begin{proof}
We will use the above two lemma's to prove this result. We verify both conditions for Lemma \ref{theorem separating set with extra condition is only positive on positive elements}. Firstly, we choose a non-zero element $a \in A$. Since $\Pi$ separates the points of $A$, there is a $\pi \in \Pi: A \to C_\pi$ such that $\pi(a)$ is non-zero. But then it is clear that there also exists a state $\omega$ on $C_\pi$ such that $\omega(\pi(a)) \neq 0$. Hence we have found a state $\omega \circ \pi$ in $S$ that is non-zero on $a$, which shows that $S$ separates the points of $A$.

Secondly, we take an element $a \in A$ and an arbitrary state $\varphi$ in $S$. Then there exists a $\pi \in \Pi: A \to C_\pi$ and a state $\omega$ on $C_\pi$ such that $\varphi = \omega \circ \pi$. Since $\pi$ is a *-homomorphism, one can compute that $\varphi_a = \omega_{\pi(a)} \circ \pi$ if $\varphi(a^*a) \neq 0$. This clearly implies that $\varphi_a$ belongs to $S$, since $\omega_{\pi(a)}$ is still a state on $C_\pi$.

So, from Lemma \ref{theorem separating set with extra condition is only positive on positive elements}, we may conclude that an element $a \in A$ is positive if and only if $\varphi(a) \geq 0$ for all $\varphi \in S$. But then Lemma \ref{theorem set of states only positive on positive elements is convex-weakly *-dense} tells us that the convex hull of $S$ is weakly *-dense in the state space $\mathcal{S}(A)$.
\end{proof}

Now we have all the results needed to prove the main Theorem \ref{theorem largest quantum subgroup acting isometrically} of this section. To be able to construct the desired Woronowicz Hopf C*-ideal, we introduce some notations.

\begin{definition} \label{definition multi-morphism of QTGs}
Let $B$ be a unital C*-algebra and take QTGs $(A, \Delta, \alpha)$ and $(C_i, \Delta_i, \alpha_i)$ for $i=1, \cdots, n$. We say that $\pi: A \to C_1 \otimes \cdots \otimes C_n$ is a \textbf{multi-morphism} of QTGs if $\pi$ is of the form $\pi_1 \ast \cdots \ast \pi_n$ for some morphisms of QTGs $\pi_i: A \to C_i$ ($i=1, \cdots, n$). We used the notation $\pi_1 \ast \pi_2$ to denote the *-homomorphism $(\pi_1 \otimes \pi_2) \Delta: A \to C_1 \otimes C_2$, and inductively one can define $\pi_1 \ast \cdots \ast \pi_n$.
\end{definition}

\begin{definition} \label{definition admissible multi-morphism of QTGs}
Let $(B,L)$ be a CQMS and $(A, \Delta, \alpha)$ a QTG of $B$. Let $\pi_i: (A, \Delta, \alpha) \to (C_{\pi_i}, \Delta_{\pi_i}, \alpha_{\pi_i})$ be  morphisms of QTGs for $i=1, \cdots, n$. We say that the multi-morphism of QTGs $\pi = \pi_1 \ast \cdots \ast \pi_n$ is \textbf{admissible} if every coaction $\alpha_{\pi_i}$ is faithful and full-isometric ($i=1, \cdots, n)$.

We denote the set of all admissible multi-morphisms of QTGs on $(A, \Delta, \alpha)$ by $\Pi_A$.
\end{definition}

\begin{proof}[Proof of Theorem \ref{theorem largest quantum subgroup acting isometrically}: ]
We claim that the desired Woronowicz Hopf C*-ideal will be
$$I = \bigcap_{\pi \in \Pi_A} \ker(\pi).$$

For the first part of the theorem, we have to prove that the set $I$ is a Woronowicz Hopf C*-ideal. It is easy to see that $I$ is a closed two-sided *-ideal, since every kernel of a *-homomorphism is a closed *-ideal. We want to show that $(p_I \otimes p_I) \Delta(I) = \{0\}$. To do so, it is sufficient to prove that $(\pi_1 \otimes \pi_2) \Delta(I) = \{0\}$ for every $\pi_1, \pi_2 \in \Pi_A$. So we choose $\pi_1 = \pi_1^{(1)} \ast \cdots \ast \pi_n^{(1)}: A \to C_1^{(1)} \otimes \cdots \otimes C_n^{(1)}$ and $\pi_2 = \pi_1^{(2)} \ast \cdots \ast \pi_m^{(2)}: A \to C_1^{(2)} \otimes \cdots \otimes C_m^{(2)}$ in $\Pi_A$. This means that $(C_i^{(j)}, \Delta_i^{(j)}, \alpha_i^{(j)})$ are faithful QTGs such that the coactions $\alpha_i^{(j)}$ are full-isometric and $\pi_i^{(j)}$ are morphisms of QTGs.

But then obviously
$$\pi_1 \ast \pi_2 = \pi_1^{(1)} \ast \cdots \ast \pi_n^{(1)} \ast \pi_1^{(2)} \ast \cdots \ast \pi_m^{(2)}$$
is still a multi-morphism of QTGs, and since all $\alpha_i^{(j)}$ are faithful and full-isometric, we know that $\pi_1 \ast \pi_2$ is admissible. Hence $(\pi_1 \otimes \pi_2) \Delta(I) = (\pi_1 \ast \pi_2)(I) = \{0\}$, so $I$ is a Worononwicz Hopf C*-ideal.

Notice that the counit $\epsilon: A \to \mathbb{C}$ is an admissible morphism of QTGs. Therefore $I$ is a proper Woronowicz Hopf C*-ideal.

Hence it makes sense to consider the CQG $A/I$ with comultiplication
$$\Delta_I: A/I \to A/I \otimes A/I: a + I \mapsto (p_I \otimes p_I) \Delta(a)$$ and the coaction $\alpha_I$ of $(A/I, \Delta_I)$ on $B$, as defined in the theorem.

Since $\alpha$ is faithful, it is clear that also the induced coaction $\alpha_I = (\iota \otimes p_I) \alpha$ is faithful. To prove that $\alpha_I$ is full-isometric, we first choose faithful QTGs $(C_i, \Delta_i, \beta_i)$ for $i=1, \cdots, n$ such that $\beta = \beta_1 \ast \cdots \ast \beta_n$ is 1-isometric. Denote $C_1 \otimes \cdots \otimes C_n$ by $C$. We want to prove that $\alpha_I \ast \beta: B \to B \otimes A/I \otimes C$ is 1-isometric, using Lemma \ref{theorem inequality 1-isometric sufficient for weak*-dense convex set of states}.

We choose a multi-morphism $\pi \in \Pi_A$. Then we can write $\pi = \pi_1 \ast \cdots \ast \pi_m$ for some morphisms $\pi_i: (A, \Delta, \alpha) \to (C_{\pi_i}, \Delta_{\pi_i}, \alpha_{\pi_i})$ of QTGs. We use $C_\pi$ to denote $C_{\pi_1} \otimes \cdots \otimes C_{\pi_m}$. We introduce some more notations: for $i \in \{1, \cdots, m\}$ we write $\bar{\pi_i}$ for the unique $\ast$-homomorphism from $A/I$ to $C_{\pi_i}$ such that $\bar{\pi_i} \circ p_I = \pi_i$. Remark that this is well-defined since $\pi_i(I) = \{0\}$. We denote $\bar{\pi_1} \ast \cdots \ast \bar{\pi_m}$ by $\bar{\pi}$, where the $\ast$-product is defined using the comultiplication $\Delta_I$. We will also write $\alpha^{(j)}$ for $\alpha \ast \cdots \ast \alpha$, the $\ast$-product of $j$ factors $\alpha$ (hence $\alpha^{(j)}$ is a mapping from $B$ to $B \otimes A^{\otimes j}$). The symbol $\Delta^{(j)}$ denotes $(\Delta \otimes \iota \otimes \cdots \otimes \iota) \cdots (\Delta \otimes \iota) \Delta$, where we have $j$ factors in the composition (hence $\Delta^{(j)}$ is a mapping from $A$ to $A^{\otimes (j+1)})$.

Notice that for any state $\omega$ on $C_\pi \otimes C$, the mapping $\omega (\bar{\pi} \otimes \iota)$ is a state on $A/I \otimes C$. Moreover, for any $b \in B$, we have
\begin{eqnarray*}
\lefteqn{L((\iota \otimes \omega (\bar{\pi} \otimes \iota)) (\alpha_I \ast \beta)(b))} \\
& = & L((\iota \otimes \omega) (\iota \otimes \bar{\pi} \otimes \iota) (\iota \otimes p_I \otimes \iota) (\alpha \otimes \iota) \beta(b)) \\
& = & L((\iota \otimes \omega) (\iota \otimes \pi \otimes \iota) (\alpha \otimes \iota) \beta (b)) \\
& = & L((\iota \otimes \omega) (\iota \otimes \pi_1 \otimes \cdots \otimes \pi_m \otimes \iota) (\iota \otimes \Delta^{(m-1)} \otimes \iota) (\alpha \otimes \iota) \beta(b)) \\
& = & L((\iota \otimes \omega) (\iota \otimes \pi_1 \otimes \cdots \otimes \pi_m \otimes \iota) (\alpha^{(m)} \otimes \iota) \beta(b)) \\
& = & L((\iota \otimes \omega) (\alpha_{\pi_1} \ast \cdots \ast \alpha_{\pi_m} \otimes \iota) \beta(b)) \\
& = & L((\iota \otimes \omega) (\alpha_{\pi_1} \ast \cdots \ast \alpha_{\pi_m} \ast \beta)(b))
\end{eqnarray*}

Now, since $\pi$ is admissible, we know that every coaction $\alpha_{\pi_i}$ is faithful and full-isometric. In particular $\alpha_{\pi_m}$ is full-isometric, so, because of the choice of $\beta$, we now know that $\alpha_{\pi_m} \ast \beta$ is 1-isometric. Since $\alpha_{\pi_m}$ is faithful and $\alpha_{\pi_{m-1}}$ is full-isometric, this implies that $\alpha_{\pi_{m-1}} \ast \alpha_{\pi_m} \ast \beta$ is 1-isometric. We can continue this argument to conclude that $\alpha_{\pi_1} \ast \cdots \ast \alpha_{\pi_m} \ast \beta$ is 1-isometric. Together with the above calculation, it follows that $L((\iota \otimes \omega (\bar{\pi} \otimes \iota)) (\alpha_I \ast \beta)(b)) \leq L(b)$ for any element $b \in B$, any $\pi: A \to C_\pi$ in $\Pi_A$ and any state $\omega$ on $C_\pi \otimes C$.

By Lemma \ref{theorem inequality 1-isometric sufficient for weak*-dense convex set of states}, we need to show that the convex hull of the set
$$S = \{\omega (\bar{\pi} \otimes \iota) \; | \; \pi \in \Pi_A: A \to C_\pi, \; \omega \in \mathcal{S}(C_\pi \otimes C)\}$$
is weakly *-dense in the set of states on $A/I \otimes C$, to conclude that $\alpha_I \ast \beta$ is 1-isometric. This will follow from applying Lemma \ref{theorem states composed with separating *-homomorphisms convex-weak *-dense} to the C*-algebra $A/I \otimes C$ and the set $\Pi$ of *-homomorphisms $\bar{\pi} \otimes \iota: A/I \otimes C \to C_\pi \otimes C$ with $\pi \in \Pi_A$. It is sufficient to prove that this set of *-homomorphisms separates the points of $A/I \otimes C$. We choose an element $x$ in $A/I \otimes C$ and suppose $x$ is nonzero. Then, since the product states separate the points of the minimal tensor product, there exists a state $\psi$ on $C$ such that $(\iota \otimes \psi)(x)$ is non-zero. We can take an element $y$ in $A$ such that $p_I(y) = (\iota \otimes \psi)(x)$. Since this element is non-zero in $A/I$, we know that $y \not\in I$. Hence, there exists a morphism $\pi$ in $\Pi_A$ such that $\pi(y)$ is non-zero. But then $\bar{\pi}(\iota \otimes \psi)(x) = \bar{\pi}(p_I(y)) = \pi(y)$ is non-zero in $C_\pi$. So we may conclude that $(\bar{\pi} \otimes \iota)(x)$ is non-zero, which proves that $\Pi$ separates the points of $A/I \otimes C$. This finishes the proof of the fact that $\alpha_I$ is full-isometric.

For the third part of the theorem, let $J$ be a second Woronowicz Hopf C*-ideal in $A$, such that the coaction
$$\alpha_J: B \to B \otimes A/J: b \mapsto (\iota \otimes p_J) \alpha(b)$$
is faithful and full-isometric. We denoted the canonical projection of $A$ onto $A/J$ by $p_J$.

We want to prove that $I \subseteq J$. Since $J$ is a Woronowicz Hopf C*-ideal, we know that $A/J$ is a compact quantum group with comultiplication
$$\Delta_J: A/J \to A/J \otimes A/J: p_J(a) \mapsto (p_J \otimes p_J)\Delta(a).$$
Hence, by definition, the mapping $p_J$ is a morphism of QTGs from $(A, \Delta, \alpha)$ to $(A/J, \Delta_J, \alpha_J)$. Moreover, since $\alpha_J$ is a faithful and full-isometric coaction, we know that $p_J$ is an admissible morphism of QTGs, and hence $p_J$ belongs to $\Pi_A$. But then, by construction, $I$ is a subset of the kernel of $p_J$, which is exactly $J$.
\end{proof}

\section{Existence of isometry groups?}

The ultimate goal of this theory would be to find something like a \textit{quantum isometry group} of a compact quantum metric space $(B,L)$. This quantum isometry group would have to be a universal object in the category of all faithful quantum transformation groups of $B$ that preserve the extra structure $L$. We already know in what way the structure should be preserved:

\begin{definition} \label{definition full-isometric QTG}
Let $(B,L)$ be a compact quantum metric space. We call a quantum transformation group $(A, \Delta, \alpha)$ of $B$ \textbf{full-isometric} if $\alpha$ is a full-isometric coaction.
\end{definition}

For a CQMS $(B,L)$, we will be considering the category $\mathcal{C}(B,L)$ of all faithful and full-isometric transformation groups $(A, \Delta, \alpha)$ of $B$. The morphisms are morphisms of transformation groups, as defined in Definition \ref{definition morphism of QTGs}

It is also possible that $B$ has some additional structure, on top of the Lipnorm. For example, we can consider a functional $\psi$ on $B$ that should be preserved by the coaction. To this purpose, we introduce a second category. For a CQMS $(B,L)$ and a functional $\psi$ on $B$, we consider the category $\mathcal{C}(B,L,\psi)$ of all faithful and full-isometric transformation groups of the pair $(B, \psi)$ as defined in Definition \ref{definition morphism of QTGs}. The morphisms are still the morphisms of transformation groups.

\begin{definition} \label{definition quantum isometry group}
The \textbf{quantum isometry group of a CQMS $(B,L)$} is the universal object in the category $\mathcal{C}(B,L)$, if such an object exists. Hence it is an object $(A,\Delta,\alpha)$ in $\mathcal{C}(B,L)$ such that, for every object $(\tilde{A}, \tilde{\Delta}, \tilde{\alpha})$ in $\mathcal{C}(B,L)$, there is a unique morphism of transformation groups from $(A, \Delta, \alpha)$ to $(\tilde{A}, \tilde{\Delta}, \tilde{\alpha})$.

If $\psi$ is a functional on $B$, we define the \textbf{quantum isometry group of $(B,L,\psi)$} to be the universal object in $\mathcal{C}(B,L,\psi)$ if such an object exists.
\end{definition}

Obviously it is not immediately clear under what conditions such a quantum isometry group exists, but Theorem \ref{theorem largest quantum subgroup acting isometrically} is certainly a step in the right direction. This theorem allows us to define the quantum isometry group if we have a 'quantum permutation group' at our disposition. For example, for finite compact quantum metric spaces, we can define a quantum isometry group fixing the trace (or another functional), since S. Wang defined the quantum automorphism group \cite{Wang-Quantum_symmetry_groups_of_finite_spaces},\cite{Wang-Ergodic_Actions_of_universal_quantum_groups_on_operator_algebras}.

\begin{theorem}
Let $(B,L)$ be a finite compact quantum metric space and let $\psi$ be a functional on $B$. Then there exists a quantum isometry group of $(B,L,\psi)$.
\end{theorem}

\begin{proof}
Denote by $(A, \Delta, \alpha)$ the quantum automorphism group of $(B,\psi)$, according to Wang's
definition \cite{Wang-Quantum_symmetry_groups_of_finite_spaces}. This means that $(A, \Delta, \alpha)$ is a universal object in the category $\mathcal{C}(B,\psi)$ of all faithful transformation groups of the pair $(B,\psi)$. Note that $(A, \Delta)$ is of Kac type, and hence has a bounded counit.

Then we claim that $(A/I, \Delta_I, \alpha_I)$ is the desired quantum isometry group of $(B, \psi)$, where $I$ is the Woronowicz Hopf C*-ideal as in (the proof of) Theorem \ref{theorem largest quantum subgroup acting isometrically}. That is, $I$ is the intersection of all kernels of admissible multi-morphisms of QTGs on $(A, \Delta, \alpha)$. As before, we used $\Delta_I$ to denote the induced comultiplication of $A/I$, and $\alpha_I$ to denote the induced coaction $(\iota \otimes p_I) \alpha$.

From Theorem \ref{theorem largest quantum subgroup acting isometrically} we already know that $(A/I, \Delta_I, \alpha_I)$ is a faithful and full-isometric quantum transformation group. To prove the universality, let $(\tilde{A}, \tilde{\Delta}, \tilde{\alpha})$ be a second faithful and full-isometric quantum transformation group of $(B,\psi)$. We want to show that there exists a unique morphism of quantum transformation groups from $(A/I, \Delta_I, \alpha_I)$ to
$(\tilde{A}, \tilde{\Delta}, \tilde{\alpha})$.

Because of the universality of the object $(A, \Delta, \alpha)$ in the category $\mathcal{C}(B,\psi)$, we know that there is a unique morphism of QTGs $\theta: A \rightarrow \tilde{A}$ from $(A,\Delta,\alpha)$ to $(\tilde{A},\tilde{\Delta},\tilde{\alpha})$.

Since $\theta$ is a morphism of QTGs, we know that $\alpha_\theta = (\iota \otimes \theta) \alpha$ equals $\tilde{\alpha}$, hence $\alpha_\theta$ is full-isometric. But then, by definition, $\theta$ is admissible. It is now clear that $\theta(I) = \{0\}$, by the construction of $I$.

We have verified that the mapping
\begin{equation} \bar{\theta}: A/I \rightarrow \tilde{A}: (a+I) \mapsto \theta(a).\end{equation}
is well defined. Remark that $\bar{\theta} p_I = \theta$. We claim that $\bar{\theta}$ is the desired morphism of quantum transformation groups from $(A/I, \Delta_I, \alpha_I)$ to $(\tilde{A},\tilde{\Delta},\tilde{\alpha})$.
\begin{itemize}
\item It is clear that $\bar{\theta}$ is a unital *-homomorphism since $\theta$ is a unital *-homomorphism.
\item We have $(\bar{\theta} \otimes \bar{\theta}) \Delta_I p_I =
(\bar{\theta} \otimes \bar{\theta}) (p_I \otimes p_I) \Delta = (\theta \otimes \theta) \Delta = \tilde{\Delta} \theta =
 \tilde{\Delta} \bar{\theta} p_I$ which means that the mappings $(\bar{\theta} \otimes \bar{\theta}) \Delta_I$
 and $\tilde{\Delta} \bar{\theta}$ are equal on $A/I$.
\item We also have $(\iota \otimes \bar{\theta}) \alpha_I = (\iota \otimes \bar{\theta}) (\iota \otimes p_I) \alpha =
(\iota \otimes \theta) \alpha = \tilde{\alpha}$.
\end{itemize}

To conclude the proof, we have to check the uniqueness of $\bar{\theta}$. But if we take any morphism of quantum transformation groups $\sigma$ from $(A/I, \Delta_I, \alpha_I)$ to $(\tilde{A}, \tilde{\Delta}, \tilde{\alpha})$, then $\sigma p_I$ would be a morphism of quantum transformation groups from $(A, \Delta, \alpha)$ to $(\tilde{A}, \tilde{\Delta}, \tilde{\alpha})$. Because of the uniqueness of $\theta$, this implies that $\sigma p_I = \theta$. But then of course $\sigma = \bar{\theta}$, which proves the uniqueness of $\bar{\theta}$.
\end{proof}

The previous theorem shows that every finite quantum metric space had an isometry group (fixing a state). We compute this isometry group in a 'small' example, where the isometry group is non-classical: the quantum space $M_2(\C) \oplus \C \oplus \C$. One can check that the quantum symmetry group of Wang (preserving the trace) reduces to the C*-algebra $A$ generated by four generators $x,y,z,p$ with relations
$$\left\{ \begin{array}{l} x^2 = -yz \\
2xx^*+yy^*+zz^* = 1 \\
p^* = p = p^2
\end{array} \right. $$
and such that the *-algebra generated by $x,y,z$ is commutative.

The comultiplication is defined by
$$ \begin{array}{rcl}
\Delta(x) & = & (zz^* - yy^*) \otimes x + x \otimes z + x^* \otimes y \\
\Delta(y) & = & (yx^* - xz^*) \otimes 2x + y \otimes z + z^* \otimes y \\
\Delta(z) & = & (xy^* - zx^*) \otimes 2x + z \otimes z + y^* \otimes y \\
\Delta(p) & = & p \otimes p + (1-p) \otimes (1-p).
\end{array} $$

We will write the elements of $M_2(\C) \oplus \C \oplus \C$ as triples. We use the notation $e_{ij} \; (i,j=1,2)$ for the matrix units in $M_2(\C)$.
Then the coaction of $A$ on $M_2(\C) \oplus \C \oplus \C$ is the unital *-homomorphism defined by
$$\begin{array}{rcl}
\alpha(e_{12},0,0) & = & (e_{11} - e_{22},0,0) \otimes x + (e_{12},0,0) \otimes z + (e_{21},0,0) \otimes y\\
\alpha(0,1,0) & = & (0,1,0) \otimes p + (0,0,1) \otimes (1-p).
\end{array}$$

There are several Lipnorms that make this space into a CQMS. The Lipnorm $L$ we will consider here will be the following:
$$L( \left( \begin{array}{cc} a & b \\ c & d \end{array} \right), e,f) = |a-d| + |b| + |c| + |a-e| + |a-f|$$
for all $a,b,c,d,e,f \in \C$.

For this Lipnorm, one can prove that a representation $\pi$ of $A$ is admissible if and only if $\pi(x) = \pi(y) = 0$. Hence, to find the quantum isometry group, we take the quotient of $A$ by the ideal generated by $x$ and $y$. We find that the quantum isometry group of $M_2(\C) \oplus \C \oplus \C$ is the universal C*-algebra generated by a unitary element $z$ and a projection $p$. The comultiplication is given by
$$ \begin{array}{rcl}
\Delta(z) & = & z \otimes z \\
\Delta(p) & = & p \otimes p + (1-p) \otimes (1-p)
\end{array} $$
and the coaction $\alpha$ is given by
$$ \begin{array}{rcl}
\alpha(e_{12},0,0) & = & (e_{12},0,0) \otimes z \\
\alpha(0,1,0) & = & (0,1,0) \otimes p + (0,0,1) \otimes (1-p).
\end{array} $$

\bibliographystyle{plain}
\bibliography{isometric-coactions}

\end{document}